\documentclass[12pt]{amsart}
\usepackage[utf8]{inputenc}
\usepackage{color}
\usepackage{xcolor}
\usepackage{comment}
\usepackage{tikz}
\usepackage{amsmath,amsfonts,euscript,amssymb,amsthm,a4,times}

\allowdisplaybreaks
\def\XXint#1#2#3{{\setbox0=\hbox{$#1{#2#3}{\int}$}
\vcenter{\hbox{$#2#3$}}\kern-.5\wd0}}

\newtheorem{theorem}{Theorem}[section]
\newtheorem{corollary}[theorem]{Corollary}
\newtheorem{lemma}[theorem]{Lemma}

\theoremstyle{definition}
\newtheorem{definition}[theorem]{Definition}

\makeatletter \catcode`@=11
\newbox\tr@tto
\setbox\tr@tto=\hbox{{\count0=0\dimen0=-,9pt\dimen1=1,1pt\loop\ifnum
    \count0<11 \advance \count0 by1 \vrule width.51pt height\dimen1
    depth\dimen0\kern-0.17pt\advance\dimen0 by-0.05pt\advance\dimen1
    by0.1pt\repeat \loop\ifnum\count0<21\advance \count0 by1 \vrule
    width.6pt height\dimen1 depth\dimen0\kern-0.2pt \advance\dimen0
    by-0.1pt\advance\dimen1 by 0.05pt\repeat}}
\def\medint{\displaystyle\copy\tr@tto\kern-10.4pt\int}
\numberwithin{equation}{section}

\def\R{{\mathbb R}}

\def\proofof#1{\begin{proof}[Proof of #1]}

\def\loc{\rm loc}

\def\R{{\mathbb R}}

\def\proofof#1{\begin{proof}[Proof of #1]}

\def\loc{\rm loc}

\catcode`@=11
\newbox\tr@tto
\setbox\tr@tto=\hbox{{\count0=0\dimen0=-,9pt\dimen1=1,1pt\loop\ifnum\count0<11 \advance \count0 by1 \vrule width.51pt height\dimen1
     depth\dimen0\kern-0.17pt\advance\dimen0 by-0.05pt\advance\dimen1
     by0.1pt\repeat \loop\ifnum\count0<21\advance \count0 by1 \vrule
     width.6pt height\dimen1 depth\dimen0\kern-0.2pt \advance\dimen0
     by-0.1pt\advance\dimen1 by 0.05pt\repeat}}
\def\medint{\displaystyle\copy\tr@tto\kern-10.4pt\int}
\catcode`@=12

\numberwithin{equation}{section}

\date{\today}

\begin{document}

%\maketitle
\title[On the interplay between $(p,q)$-growth and the $x$-dependence]{On the interplay between $(p,q)$-growth and \\ $x$-dependence of the energy integrand:\\ a limit case}
\author{M. Eleuteri, P. Marcellini, E. Mascolo $\&$ A. Passarelli di Napoli}

\address{Michela Eleuteri , Dipartimento di Scienze Fisiche, Informatiche e Matematiche  \\ Universit\`{a} di Modena e Reggio Emilia, via Campi 213/b - 41121 Modena, Italy}
\email{michela.eleuteri@unimore.it}

\address{Paolo Marcellini, Dipartimento di Matematica e Informatica \lq\lq U. Dini"\\
Universit\`a di Firenze,
Viale Morgagni 67/a,
50134 Firenze,
Italy} \email{paolo.marcellini@unifi.it}

\address{Elvira Mascolo, Dipartimento di Matematica e Informatica \lq\lq U. Dini"\\
Universit\`a di Firenze,
Viale Morgagni 67/a,
50134 Firenze,
Italy} \email{elvira.mascolo@unifi.it}

\address{Antonia Passarelli di Napoli, Dipartimento di Matematica e Applicazioni "R.
Caccioppoli" \\ Universit\`{a} di Napoli ``Federico
II", via Cintia - 80126 Napoli, Italy}
\email{antpassa@unina.it}

\maketitle
\begin{center}
{\footnotesize{\textit{This paper is dedicated to Gioconda Moscariello  on the
			occasion of her 70th birthday.  %\\in recognition of her relevant contributions to regularity theory in \\ Partial Differential  Equation and of Calculus of Variations.
            }}}
\end{center}

\begin{abstract}
{We establish the local Lipschitz regularity of the local minimizers of non autonomous integral functionals of the form
\[
\int_\Omega F(x, Dz)\,dx,
\]
where $\Omega$ is a bounded open set of $\mathbb{R}^n$, $n \ge 2$. The energy density $F(x,\xi)$  satisfies $(p,q)-$growth conditions with respect to the gradient variable and,  as a function of the $x$ variable, belongs to the Sobolev class $W^{1,\phi}$, with $\phi(t)=t^r\log^\alpha(e+t),$  $r\ge n$, $\alpha\ge 0$ under the condition
$$
1\le\frac{q}{p} \le 1 + \frac{1}{n} - \frac{1}{r}.
$$
We present a unified approach that covers the limit case 
$$
\frac{q}{p} = 1 + \frac{1}{n} - \frac{1}{r}
$$
and recovers the results in \cite{EMM16} and in \cite{CGHPdN20}.}

\end{abstract}



\bigskip

\section{Introduction}

{The study of the local Lipschitz continuity of the minimizers of variational  integrals of the form
 \begin{equation}
\label{functional}
 \int_\Omega F(x, Dz)\,dx,
\end{equation}
where $\Omega$ is a bounded open set of $\mathbb{R}^n$, $n \ge 2$, and where the energy density $F(x,\xi)$ exhibits $(p,q)-$growth conditions with respect to the gradient variable, has become a well-established topic in the literature. The study of such problems can be traced back to the pioneering contributions of Marcellini \cite{M89, M91}. Nowadays, the existing research on the regularity theory for local minimizers of functionals with non-standard growth conditions is so extensive that providing a complete list of references is unfeasible. Therefore, we limit ourselves to refer the interested reader to several survey articles on the subject \cite{Mar21, M06, Mingione-Radulescu} and to {some} latest achievements in {\cite{BM20, BS20, CM15, CMMP2023, CMMEng, CMM04, CMMpress, DFM20, DFM21, DiMMarc, EMM16, EMM3, EMM2022, EPdN23, EN23, EPT25, EPT26, Mar23, moscariello-pascale,moscariello-pascale2}.}
\\
To avoid the occurrence of the so-called {\it Lavrentiev phenomenon} — a well-known obstacle to the regularity of the minimizers (see  the recent works \cite{BDS, FDFL}) — the typical approach  is to identify suitable conditions on the partial map $x \mapsto F(x, \xi)$ that ensure a certain degree of regularity for the minimizers.
\\
A first contribution in this direction was given by Eleuteri-Marcellini-Mascolo in \cite{EMM16} (see also \cite{EMM3}), where, under the assumption that the partial map $x \to F(x,\xi)$ belongs to the Sobolev space $W^{1,r}$, with $r > n$, and that the map $\xi \to F(x,\xi)$ satisfies $(p,q)$-growth conditions, a sufficient criterion for the Lipschitz regularity of the minimizers is provided. This condition involves a closeness  requirement on the exponents $(p,q)$, which depends on both $n$ and $r$, and it is given by
\begin{equation}
\label{gap}
\frac{q}{p} < 1 + \frac{1}{n} - \frac{1}{r}.
\end{equation}
Condition \eqref{gap} has been taken into consideration in various similar contexts; see for example DeFilippis-Mingione \cite{DFM20}.
{The proof of the Lipschitz continuity result in \cite{EMM16} (see also \cite{EMM3}) crucially relies on the strict inequality \eqref{gap}, therefore it is natural to ask whether the same result can be achieved by approaching the limit case. 
This is precisely the purpose of this contribution. In particular we will  provide a unified approach to the original situation and the limit case, also covering the standard growth case \cite{CGHPdN20}}.
\\
To  describe our results  precisely, we  now introduce the assumptions.
We  assume that the energy density $F(x, \xi)$ is a Carath\'eodory function and  that there exists $\tilde F:\Omega\times [0,+\infty)\mapsto [0,+\infty)$ such that $F(x, \xi)= \tilde F(x,|\xi|)$. We assume moreover that $F_{\xi x}$ is a  Carath\'eodory function, that $F(x,0)=0$ and that there exist a couple of exponents {$1<p\le q$},  constants $\lambda,\Lambda>0$, a parameter $\mu\in [0,1]$ and a nonnegative measurable function $h(x)$ such that 
%\[
%(\mu^2 + |\xi|^2)^{p/2} \le \, F(x, \xi) = \tilde{F}(x, |\xi|) \le ( + |\xi|^2)^{q/2} 
%\eqno{\rm (F1)}\]

\begin{equation}\label{(F2)(F3)}\lambda \, (\mu^2 + |\xi|^2)^{\frac{p-2}{2}} |\zeta|^2 \le \sum_{i,j = 1}^n F_{\xi_i \xi_j} (x, \xi) \zeta_i \zeta_j\le \Big(1 + \Lambda \, (\mu^2 + |\xi|^2)^{\frac{q-2}{2}}\Big)|\zeta|^2, \end{equation}
%\[
%|F_{\xi_i \xi_j}(x, \xi)| \le \, \Lambda \, (\mu^2 + |\xi|^2)^{\frac{q-2}{2}} , \eqno{\rm (F3)} \]

\begin{equation}\label{(F4)}|F_{\xi x}(x, \xi)| \le \, h(x) \, (\mu^2 + |\xi|^2)^{\frac{q-1}{2}},  \end{equation}
for every $\xi,\zeta \in \mathbb{R}^n$ and a.e. $x$ in $\Omega$. Here we supposed that $\xi\to F(x,\xi)$ is of class $\mathcal{C}^2(\mathbb{R}^n)$ if $\mu>0$ and is of class $\mathcal{C}^2(\mathbb{R}^n\setminus\{0\})$ if $\mu=0$. 
As proved in \cite[Lemma 2.3]{CMMP2023}, assumption \eqref{(F2)(F3)} entails the following  $(p,q)$ growth
conditions for the energy density $F (x, \xi)$
\begin{equation}\label{F1}
c(\lambda,p) \, (\mu^2 + |\xi|^2)^{\frac{p-2}{2}} |\xi|^2 \le F (x, \xi) \le C(\Lambda,q)\,( 1 + |\xi|^2)^{\frac{q-2}{2}}|\xi|^2.
\end{equation}
%\colorbox{yellow}{forse qui a destra nella (1.5) basta un solo +1? SI! ho corretto}

Regarding the function $h(x)$ appearing in  \eqref{(F4)}, we  assume that there exist
%a  strictly increasing function 
%{$\mathcal{L}:[0,+\infty]\to [0,+\infty]$  
%and  
exponents $r\ge n$ and $\alpha\ge 0$ such that 
\begin{equation}\label{F5F6}
\
 \int_{B_R} h^r \log^{\alpha } (e + h)  \, dx < + \infty, \qquad \forall B_R \Subset \Omega. 
\end{equation}
The previous assumption  implies that the function $h(x)$  in \eqref{(F4)} belongs to the Orlicz-Zygmund space $L^r \log^{\alpha }L$. %which reduces to $L^r$ when $p < q,$ in agreement with \cite{EMM16}. 
Hence,  the class of functions describing the regularity of the partial map $x \mapsto F(x, \xi)$ is clearly contained in  $W^{1,n}$, consistently with the fact that, even in case of standard growth  conditions (that corresponds to the case $p = q$ in assumption \eqref{F1}), the regularity can be obtained assuming a $W^{1,{L^n\log^\alpha L}}$ regularity of the coefficients, with $\alpha$ sufficiently large with respect to the dimension $n$ (see \cite{CGHPdN20} and \cite{EPdN23}  for the variable growth case). In this respect, our result also retrieves the corresponding one in  \cite{CGHPdN20}, which in turn is a refinement of \cite{EMM19}.

\vspace{2mm}

Let us recall the definition of local minimizer
\begin{definition}
A mapping $u\in W^{1,1}_{\rm loc}(\Omega)$ is a local minimizer {of the integral functional \eqref{functional}} if
$F(x,Du) \in L^{1}_{\rm loc}(\Omega )$ and
\begin{equation*}\label{minineq}
 \int_{\mathrm{supp}\varphi} \! F(x,Du) \, dx \leq \int_{\mathrm{supp}\varphi} \! F(x,Du+D\varphi) \, dx   
\end{equation*}
for  any $\varphi\in W^{1,1}(\Omega)$ with $\mathrm{supp}\varphi\Subset\Omega$.
\end{definition} 

%{The main result of our paper is the following.}

We formulate our main results in the next Theorems \ref{mainresult},  \ref{mainresult2} and \ref{mainresult3}. In particular, in the first and the second ones, we focus on the limit case with equality in \eqref{gap}.

\begin{theorem}
\label{mainresult}
    Let $u\in W^{1,p}_{\mathrm{loc}}(\Omega)$ be a local minimizer of \eqref{functional} under the assumptions \eqref{(F2)(F3)}, \eqref{(F4)}.
  Assume that  the equality holds in \eqref{gap}, i.e. 
   \begin{equation}\label{definitionr}
      \frac{q}{p} = 1+\frac{1}{n}-\frac{1}{r}.  
    \end{equation} 
 If \eqref{F5F6} holds  with  $r>n$ and $\alpha>0$, 
%
 %  $p=q$  and
 %$h(x)$ satisfies \eqref{F5F6}  with $r=n$ and $\alpha>4$;\\ 
 %or
% { ii)}  $p<q$,  $\displaystyle \frac{q}{p} < 1+\frac{1}{n}-\frac{1}{r}$, and $h(x)$ satisfies \eqref{F5F6} with $r>n$ and $\alpha=0$; 
%\\
% or
%{ iii)} 
%$p<q$,  $\displaystyle \frac{q}{p} = 1+\frac{1}{n}-\frac{1}{r}$ and $h(x)$ satisfies \eqref{F5F6} with $r>n$
% and $\alpha>0.$
 %\\
 %Of course, these conditions are mutually exclusive.
then $u\in W^{1,\infty}_{\mathrm{loc}}(\Omega)$ and for every pair of concentric balls $B_\rho\subset B_R\Subset\Omega,$ the following estimate
 \begin{eqnarray}
     \label{mainaptotale}
||Du||_{L^{\infty}(B_{\rho}; {\mathbb{R}^n)}} &\le& G\left( C\int_{B_{R}} \!\!\Big(1 + {F(x,Du)} \Big)
\, dx\right)\cr\cr 
&&\qquad + \frac{ C }{( R -  \rho)^{\frac{n}{p-n(q-p)}}}  \left(\int_{B_{  R}}\Big(1 + {F(x,Du)} \Big) dx\right)^{\frac{1}{p-n(q-p)}}
\end{eqnarray}
holds true  with $C \equiv C(n, r, p, q, \lambda, \Lambda, h, \alpha)$ and where %the strictly increasing function $G(t)$ is defined as follows
%for $L(t)=t^{r-n}\log^\alpha(e+t)$, 
$G(t)$ is the function defined as follows
\begin{equation*}
\label{DefinizioneGprima}
G(t)=
%\left \{
%\begin{array}{lll} %\!\!\!\!\!\!&t^{\frac1p}\qquad\qquad\qquad\qquad\qquad\qquad \qquad &\text{in the case { i)}}\\
 %   \!\!\!\!\!\! & t^\frac{r-n}{p(r-n)-rn(q-p)} + t^{\frac1p}\,\,\quad\qquad\qquad\qquad\qquad & \text{in the case { ii)}}\\
\exp\Big(\frac{r}{2p}\cdot t^{\frac{2(r-n)}{\alpha n}}\Big)t^{\frac{r}{np}}. %\qquad\quad & \text{in the case { iii)}},
%\end{array}
%\right.
\end{equation*}
 \end{theorem}

A further limiting regime  arises if $r \rightarrow n$. Then in \eqref{definitionr} also $q$ converges to $p$ and we obtain the following limit case, recovering the results in \cite{CGHPdN20}.

\begin{theorem}
\label{mainresult2}
    Let $u\in W^{1,p}_{\mathrm{loc}}(\Omega)$ be a local minimizer of \eqref{functional} under the assumptions \eqref{(F2)(F3)}, \eqref{(F4)}. % and \eqref{F5F6}.} 
  Let $r = n$ and $p = q$ so that  equality  \eqref{definitionr} holds. If \eqref{F5F6} is true with $\alpha > 4n,$ then $u\in W^{1,\infty}_{\mathrm{loc}}(\Omega)$ and for every pair of concentric balls $B_\rho\subset B_R\Subset\Omega,$ the following estimate
 \begin{eqnarray}
     \label{mainaptotalebis}
||Du||_{L^{\infty}(B_{\rho}; {\mathbb{R}^n})} &\le& \left(\!\frac{C}{( R - \rho )^{n}}  \int_{B_{R}} \!\!\Big(1 + |Du|^p \Big)
\, dx\right)^{\frac{1}{p}}
\end{eqnarray}
holds true  with $C \equiv C(n, r, p, q, \lambda, \Lambda, h, \alpha).$
\end{theorem}
In the next Theorem, we emphasize the case with strict inequality in \eqref{gap}; this result  recovers that in \cite{EMM16}.

\begin{theorem}
\label{mainresult3}
  Let $u\in W^{1,p}_{\mathrm{loc}}(\Omega)$ be a local minimizer of \eqref{functional} under the assumptions \eqref{(F2)(F3)}, \eqref{(F4)} and \eqref{F5F6} with $\alpha = 0$.
  Assume that also \eqref{gap} holds, i.e. %$r\ge n$ and $q \ge p > 1$ satisfy the inequality 
   \begin{equation}\label{definitionrter}
    \frac{q}{p} < 1+\frac{1}{n}-\frac{1}{r}.  
    \end{equation} 
 Then $u\in W^{1,\infty}_{\mathrm{loc}}(\Omega)$ and for every pair of concentric balls $B_\rho\subset B_R\Subset\Omega,$ the following estimate
 \begin{eqnarray}
     \label{mainaptotaleter}
||Du||_{L^{\infty}(B_{\rho}; {\mathbb{R}^n})} &\le& G\left(C  \int_{B_{R}} \!\!\Big(1 + |Du|^p \Big)dx\right)\cr\cr 
&&\qquad + \frac{ C }{( R -  \rho)^{\frac{n}{p-n(q-p)}}}  \left(\int_{B_{  R}}\Big(1 + |Du|^p \Big) dx\right)^{\frac{1}{p-n(q-p)}}
\end{eqnarray}
holds true  with $C \equiv C(n, r, p, q, \lambda, \Lambda, h)$ and where $G(t)$ is defined as follows
\begin{equation*}
\label{DefinizioneG3}
G(t)= t^\frac{r-n}{p(r-n)-rn(q-p)}.
\end{equation*}

%\colorbox{yellow}{scrivere l'esponente solo in termini di $p, q$ oppure solo in termini di $r$ e $n$ }\\
%\colorbox{yellow}{\tcv{non si può in questo caso. Non vale l'uguaglianza!}} \\
%\colorbox{yellow}{VALORE QUANTIFICATO DEL PARAMETRO $\alpha$}

\end{theorem}

\color{black}
The proof of the previous Theorems is achieved
through the following steps:
\begin{itemize}
\item The first step consists in an \textit{a priori} estimate for \textit{sufficiently regular minimizers} of the functional at \eqref{functional} with constants depending only on the assumptions \eqref{(F2)(F3)}, \eqref{(F4)} and \eqref{F5F6}.
%\\
%\colorbox{yellow}{qui metterei (1.2), (1.3), (1.4), (1.6)}
\vskip.2cm
\item The second step consists in constructing a suitable  approximation of the energy integral  in \eqref{functional} by means of a sequence of regular functionals  with a unique \textit{sufficiently regular}  minimizers $u_h$.
\vskip.2cm
\item  The final  step consists in applying the a priori estimate to each  $u_h$  and in showing that the sequence $(u_h)$  converges to a  minimizer of the original functional, which preserves the regularity properties. 
\end{itemize}
The originality of our result mainly relies in the first step, where  the right condition on the interplay between the regularity of the coefficients and the growth assumption is the main issue. Indeed, this is where  the extra regularity of the coefficients  beyond the $W^{1,n}$ integrability needs to compensate the rate of growth  at $\infty$ of the ellipticity ratio.
Once this first step is achieved, the Uhlenbeck structure {(\cite{U77})} of the integrands allows us to construct the approximating problems following a quite standard procedure.

A generalization of such an interplay between the nonstandard growth and the regularity of the partial map $x\to F(x,\xi)$ will be presented in a forthcoming paper \cite{EPdN25}; {instead the case of nearly linear growth has been considered in the paper \cite{EPdN26}.}

\section{Preliminary results}
\medskip
In this paper we shall denote by $C$  a
general positive constant that may vary on different occasions, even within the
same line of estimates.
Relevant dependencies  will be suitably emphasized using
parentheses or subscripts.  In what follows, $B(x,r)=B_r(x)=\{y\in \R^n:\,\, |y-x|<r\}$ will denote the ball centered at $x$ of radius $r$.
We shall omit the dependence on the center and on the radius when no confusion arises.
We recall the following well known iteration lemma, whose
proof can be found, e.g. in \cite[Lemma 6.1, p.191]{Giusti}.

  \begin{lemma}\label{lem:Giaq}
    For $0<R_1<R_2$, consider a bounded function
    $f:[R_1,R_2]\to[0,\infty)$ with
    \begin{equation*}
      f(r_1)\le\vartheta f(r_2)+\frac A{(r_2-r_1)^\alpha}+ \frac B{(r_2-r_1)^\beta}+ C
       \qquad\mbox{for all }R_1\le r_1<r_2\le R_2,
    \end{equation*}
    where $A,B,C$ denote nonnegative constants, {$\alpha > \beta > 0$ 
    and $\vartheta\in [0,1)$.} Then we have
    \begin{equation*}
      f(R_1)\le c(\alpha,\beta,\vartheta)
      \bigg(\frac A{(R_2-R_1)^\alpha}+\frac B{(R_2-R_1)^\beta}+C
      \bigg).
    \end{equation*}
  \end{lemma}
  We shall use the elementary inequality that is contained in the following 
\begin{lemma}\label{elem1} Let $\varphi:{(1,+\infty)}\to [0,+\infty)$ be the function defined by
$$ \varphi(t)=\frac{t^\beta}{\log^{\alpha}(t)},$$
with $\alpha,\beta>0$. Then
$$\varphi(st)\le 2^{\alpha}\sqrt{\varphi(s^2)\varphi(t^2)}.$$
\end{lemma}
\begin{proof} One can explicitly calculate
\begin{eqnarray*}
  \varphi(st)&=& {\frac{(st)^\beta}{\log^{\alpha}(st)}}\cr\cr 
  &=&\frac{s^\beta}{(\log(s)+\log(t))^{\frac{\alpha}{2}}}\frac{t^\beta}{(\log(s)+\log(t))^{\frac{\alpha}{2}}}\cr\cr 
  &=&\frac{s^\beta}{\left(\frac12\log(s^2)+\frac12\log(t^2)\right)^{\frac{\alpha}{2}}}\frac{t^\beta}{\left(\frac12\log(s^2)+\frac12\log(t^2)\right)^{\frac{\alpha}{2}}}\cr\cr 
  &\le &2^{\alpha}\sqrt{\varphi(s^2)\varphi(t^2)} 
\end{eqnarray*}
and this concludes the proof.
\end{proof}
\color{black}
For further needs, we recall that for a non-trivial function $\Phi:[0,+\infty)\to [0,+\infty)$,
its polar (or Fenchel conjugate)  is defined by
\begin{equation}
\label{Fconiug}
\Phi^*(s) := \sup_{t \ge 0} \big \{st - \Phi(t) \big \} \qquad {\forall \,}s\ge 0.
\end{equation}
The function $\Phi^*: [0,+\infty) \rightarrow [0,+\infty)$ is convex and   the Young-type (or Fenchel) inequality   directly follows from its  definition
\begin{equation}
\label{Fenchel}
st \le \, \Phi^*(s) + {\Phi(t)}
\end{equation}
for all $s,t\ge 0$.
\\
We conclude by recalling the  following
\begin{lemma}\label{GGP} Let $\varphi:[0,+\infty)\to [0,+\infty)$ be  increasing and such 
\begin{equation}\label{assumlemma}\lim_{t\to 0}\frac{\varphi(t)}{t}=0\qquad\qquad \lim_{t\to +\infty}\frac{\varphi(t)}{t}=+\infty.\end{equation}Then there exists $s_0>0$ such that 
$$\frac{s}{\varphi^{-1}(s)}\le \psi(s)\le 2\frac{s}{\varphi^{-1}(s)},\qquad \text{for every } s>s_0$$
where  $\psi(t)=(\varphi^*)^{-1}(t)$.
\end{lemma}
For the proof we refer to \cite[Lemma 2.4]{GGP}, see also \cite[Lemma 2.1]{GianGrePas}.
\section{A Caccioppoli type inequality}
This section is devoted to the proof of a suitable second order Caccioppoli inequality that is the first step in the Moser iteration procedure. Even though such inequalities are available in the literature in slightly different forms, we give the proof here for the sake of completeness. The proof is based on the a priori assumptions $u\in W^{1,\infty}_{\loc}(\Omega)$ and $(\mu^2+|Du|^2)^{\frac{p-2}{4}}Du\in W^{1,2}_{\loc}(\Omega)$ which allow us to  test the differentiated Euler–Lagrange equation with functions that are proportional to the gradient of the minimizers.
\begin{lemma}
Let $u\in W^{1,\infty}_{\loc}(\Omega)$ such that $(\mu^2+|Du|^2)^{\frac{p-2}{4}}Du\in W^{1,2}_{\loc}(\Omega)$ be  a local minimizer of the functional at \eqref{functional} under the assumptions \eqref{(F2)(F3)} and \eqref{(F4)}. Then the following second order Caccioppoli type inequality
\begin{eqnarray}\label{Caccioppoli}
 && \int_{\Omega} \eta^2 (\mu^2 + |Du|^2)^{\frac{p-2+\gamma}{2}}  |D^2 u|^2 \, dx\cr\cr
&\le & C (1+\gamma^2)  \, \int_{\Omega} {\eta^2} h^2(x)(\mu^2 + |Du|^2)^{\frac{2q-p+\gamma}{2} }  \, dx\cr\cr
&&+ \, C \int_{\Omega} |D \eta|^2 \,\Big(1 +  (\mu^2 + |Du|^2)^{\frac{q+\gamma}{2}} \Big) \, dx,
\end{eqnarray}
holds true with $C=C(\lambda,\Lambda,n)$, for every $\gamma \ge \,0,\,$ and for every $\eta\in C^1_0(\Omega).$
\end{lemma}
\color{black}
\begin{proof}
By the assumptions 
\begin{equation}
\label{regE}
u \in  W^{1,\infty}_{\loc}(\Omega)\,\, \text{and}  \,\, (\mu^2+|Du|^2)^{\frac{p-2}{4}}Du\in W^{1,2}_{\loc}(\Omega; {\mathbb{R}^n}),
\end{equation}
we have that the following system
\begin{equation}
\label{second_variation}
\int_{\Omega} \left (\sum_{i,j=1}^n  F_{\xi_i \xi_j}(x, Du)u_{x_j x_s}D_{x_i} \varphi  + \sum_{i=1}^n F_{\xi_i x_s}(x, Du) D_{x_i} \varphi \right ) \, dx = 0,
\end{equation}
holds for all $s = 1, \dots, n$ and for all $\varphi \in W^{1,p}_0(\Omega)$. Let $\eta \in \mathcal{C}^1_0(\Omega)$ and test \eqref{second_variation} with $\varphi=\eta^2
({\mu}^2+|Du|^2)^{\frac{\gamma}{2}} u_{x_s}$, for some $\gamma \ge 0$ so that
\begin{eqnarray*}
 D_{x_i} \varphi &=&2 \eta \eta_{x_i} ({\mu}^2 +|Du|^2)^{\frac{\gamma}{2}} u_{x_s}+ \eta^2 ({\mu}^2 +|Du|^2)^{\frac{\gamma}{2}} u_{x_s x_i}  \\
 &&+  \eta^2 \gamma ({\mu}^2 +|Du|^2)^{\frac{\gamma}{2}-1} |Du| D_{x_i}(|Du|) u_{x_s}.  
\end{eqnarray*} 
Inserting in \eqref{second_variation} we get
\begin{eqnarray*}
0 &=& 2\int_{\Omega} \sum_{i,j=1}^n F_{\xi_i \xi_j}(x, Du) u_{x_j x_s} \eta \eta_{x_i} ({\mu}^2 +|Du|^2)^{\frac{\gamma}{2}} u_{x_s} \, dx\\
&& +  \int_{\Omega} \sum_{i,j=1}^n F_{\xi_i \xi_j}(x, Du) u_{x_j x_s} \eta^2 ({\mu}^2 +|Du|^2)^{\frac{\gamma}{2}} u_{x_s x_i} \, dx\\
&& + \gamma \int_{\Omega} \sum_{i,j=1}^n F_{\xi_i \xi_j}(x, Du) u_{x_j x_s} \eta^2   ({\mu}^2 +|Du|^2)^{\frac{\gamma}{2}-1} |Du| D_{x_i}(|Du|) u_{x_s} \, dx\\
&& + 2\int_{\Omega} \sum_{i=1}^n F_{\xi_i x_s}(x, Du) \eta \eta_{x_i}  ({\mu}^2 +|Du|^2)^{\frac{\gamma}{2}} u_{x_s} \, dx\\
&& +  \int_{\Omega} \sum_{i=1}^n F_{\xi_i x_s }(x, Du) \eta^2 ({\mu}^2 + |Du|^2)^{\frac{\gamma}{2}}  u_{x_s x_i} \, dx\\
&& +  \gamma\int_{\Omega} \sum_{i=1}^n F_{\xi_i x_s}(x, Du) \eta^2   ({\mu}^2 + |Du|^2)^{\frac{\gamma}{2}-1}  |Du| D_{x_i}(|Du|) u_{x_s} \, dx .
\end{eqnarray*}
Let us sum in the previous  equation  all terms with respect to $s$ from 1 to $n$, and let us denote by ${I_1-I_6}$ the corresponding integrals.
\\
Previous equality yields
\begin{equation}\label{start}
	I_2+I_3\le |{I}_1|+|{I}_4|+|{I}_5|+|{I}_6|.
\end{equation}
Using  the left inequality in assumption \eqref{(F2)(F3)} and  that $D_{x_j}(|Du|)|Du|={\sum_{s=1}^n u_{x_j x_s} u_{x_s}}$, we can estimate the term $I_3$ as follows: 
\begin{eqnarray*}\label{I3}
{I}_3 &=& \gamma\int_{\Omega}  \sum_{i,j,s =1}^n F_{\xi_i \xi_j}(x, Du) u_{x_j x_s} \left [\eta^2   ({\mu}^2 + |Du|^2)^{\frac{\gamma}{2}-1}  D_{x_i}(|Du|)|Du| \right ] u_{x_s} \, dx\cr\cr
%&\ge& 2 \gamma \, \int_{\Omega} \eta^2 |Du|^{2\gamma-1} \sum_{i,j,s =1}^n f_{\xi_i \xi_j}(x, Du) D_{x_i}(|Du|) u_{x_j x_s} u_{x_s} \, dx\\
%&=& 2 \gamma \, \int_{\Omega} \eta^2 |Du|^{2\gamma-1} \sum_{i,j =1}^n f_{\xi_i \xi_j}(x, Du) D_{x_i}(|Du|) \left(\sum_{s=1}^n u_{x_j x_s} u_{x_s} \right)\, dx\\
&\ge&  \gamma \, \int_{\Omega} \eta^2 ({\mu}^2+|Du|^2)^{\frac{\gamma}{2}-1} {|Du|^2} \sum_{i,j =1}^n F_{\xi_i \xi_j}(x, Du) D_{x_i}(|Du|) D_{x_j}(|Du|) \, dx\cr\cr
&\ge&  \gamma \, {\lambda} \, \int_{\Omega} \eta^2 ({\mu}^2+|Du|^2)^{\frac{\gamma}{2}-1} {|Du|^2}  |D(|Du|)|^2  ({\mu}^2 + |Du|^2)^{\frac{p-2}{2}}  \, dx \ge 0.
\end{eqnarray*} 
%\textcolor{teal}{mi sembra che ci manchi un quadrato, ma controlla anche tu. Non cambia il risultato alla fine.}\textcolor{magenta}{Si. Hai ragione. Perciò ho tolto il rosso}
Therefore, the non negative term $I_3$ can be discarded in the left hand side of  \eqref{start} thus getting
\begin{eqnarray}\label{ristart}
	I_2\le |{I}_1|+|{I}_4|+|{I}_5|+|{I}_6|.
\end{eqnarray}
By  Cauchy-Schwarz and Young's inequalities and  the right inequality in assumption  \eqref{(F2)(F3)}, we have\\
\begin{eqnarray}\label{I1}
|{I}_1| &=& 2\left |\int_{\Omega}  \eta  ({\mu}^2 + |Du|^2)^{\frac{\gamma}{2}}   \sum_{i,j,s =1}^n F_{\xi_i \xi_j}(x, Du) u_{x_j x_s} \eta_{x_i} u_{x_s} \, dx\right |\cr\cr
%& = &
%2\Bigg| \int_{\Omega}  \eta  (1 + |Du|^2)^{\gamma} \cr\cr
%&& \times \sum_{s=1}^n \left( \sum_{i,j=1}^n F_{\xi_i \xi_j}(x, Du) \eta_{x_i} \eta_{x_j} u_{x_s}^2 \right)^{1/2} \left(\sum_{i,j=1}^n
%F_{\xi_i \xi_j}(x, Du) u_{x_s x_i} \, u_{x_s x_j} \right)^{1/2} \, dx \Bigg| \cr\cr
&\le& 2\int_{\Omega}  \eta  ({\mu}^2 + |Du|^2)^{\frac{\gamma}{2}} \cr\cr
&& \times \left \{ \sum_{i,j,s =1}^n F_{\xi_i \xi_j}(x, Du) \eta_{x_i} \eta_{x_j} u_{x_s}^2\right \}^{1/2} \, \left \{ \sum_{i,j,s =1}^n F_{\xi_i \xi_j}(x, Du) u_{x_s x_i} \, u_{x_s x_j}\right \}^{1/2} \, dx \cr\cr
&\le &  C   \int_{\Omega}  ({\mu}^2 + |Du|^2)^{\frac{\gamma}{2}}  \sum_{i,j,s =1}^n F_{\xi_i \xi_j}(x, Du) \eta_{x_i} \eta_{x_j} u_{x_s}^2 \, dx\cr\cr
&&  + \frac{1}{2} \int_{\Omega} \eta^2  ({\mu}^2 + |Du|^2)^{\frac{\gamma}{2}}  \sum_{i,j,s =1}^n F_{\xi_i \xi_j}(x, Du) u_{x_s x_i} \, u_{x_s x_j}  \, dx \cr\cr
&\le& {C(\Lambda)} \int_{\Omega} |D \eta|^2 \, \Big(1 + {(\mu^2 + |Du|^2)^{\frac{q + \gamma}{2}}}\Big)\, dx \cr\cr
&& + \frac{1}{2} \int_{\Omega} \eta^2  ({\mu}^2 + |Du|^2)^{\frac{\gamma}{2}}  \, \sum_{i,j,s =1}^n F_{\xi_i \xi_j}(x, Du) u_{x_j x_s} u_{x_i x_s} \, dx.
\end{eqnarray}
We  estimate the integrals $I_4$ and $I_5$ in \eqref{ristart} by the use of  Cauchy-Schwarz and  Young's inequalities, together with \eqref{(F4)}, as follows
\begin{comment}
 \textcolor{teal}{ho aggiunto il valore assoluto alla prima riga; la costante nella F4 non c'è; la costante in Cauchy-Schwarz mi pare che non ci sia}   
\end{comment}
\begin{eqnarray}\label{I4}
|{I}_4| &=& \left | 2 \int_{\Omega} \eta  (\mu^2 + |Du|^2)^{\frac{\gamma}{2}}  \sum_{i,s =1}^n F_{\xi_i x_s}(x, Du) \eta_{x_i} u_{x_s} \, dx \right |\cr\cr
&{\le}&  2\,  \int_{\Omega} \eta h(x) \, {(\mu^{2} + |Du|^2)^{\frac{q-1+\gamma}{2}}} \sum_{i,s=1}^n |\eta_{x_i} u_{x_s} | \, dx \cr\cr
&\le&   2\,  \int_{\Omega} \eta |D\eta| |Du| \, h(x) \, ({\mu^2} + |Du|^2)^{\frac{{q-1+\gamma}}{2}}  \, dx \cr\cr
&\le& 2\int_{\Omega} \eta|D\eta| h(x) {(\mu^2 + |Du|^2)^{\frac{q+\gamma}{2}}}\,dx\cr\cr
&=&  2\, \int_{\Omega} \left[|D \eta|^2 (\mu^2+|Du|^2)^{\frac{p+\gamma}{2}} \right]^{\frac{1}{2}} \left[\eta^2 ({\mu^2}+|Du|^2)^{\frac{{2 q - p+\gamma}}{2}}h^2(x)\right]^{\frac{1}{2}} \, dx 
\cr\cr
&\le&  \frac{1}{2} \, \int_{\Omega} |D \eta|^2 (\mu^2+|Du|^2)^{\frac{p+\gamma}{2}} \, dx\cr \cr
&& +  C \, \int_{\Omega} \eta^2\, h^2(x)(\mu^2+|Du|^2)^{\frac{{2 q - p+\gamma}}{2}} \, dx. 
\end{eqnarray}
Moreover
\begin{eqnarray}\label{I5}
|{I}_5| &=& \left| \int_{\Omega} \eta^2 ({\mu^2} + |Du|^2)^{\frac{\gamma}{2}} \sum_{i,s=1}^n F_{\xi_i x_s }(x, Du)  u_{x_s x_i} \, dx \right| \cr\cr
%&{\le}&  \int_{\Omega} \eta^2\,h(x) ({\mu^2} + |Du|^2)^{\frac{{q-1+\gamma}}{2}}  {\left | \sum_{i,s=1}^n  u_{x_s x_i} \right |}\, dx \cr\cr
&\le&\,  \, \int_{\Omega} \eta^2 h(x) ({\mu^2} + |Du|^2)^{{\frac{q-1+\gamma}{2}}} |D^2 u| \, dx \cr\cr
&=& \int_{\Omega} \left[\eta^2 ({\mu^2}+|Du|^2)^{\frac{p-2+\gamma}{2}} |D^2 u|^2 \right]^{\frac{1}{2}} \left[\eta^2\,h^2(x) ({\mu^2}+|Du|^2)^{\frac{{2 q - p+\gamma}}{2}}\right]^{\frac{1}{2}} \, dx \cr\cr
&\le& \varepsilon \int_{\Omega} \eta^2 ({\mu^2} + |Du|^2)^{\frac{p-2+\gamma}{2}}  |D^2 u|^2 \, dx\\\nonumber&& + C_\varepsilon  \int_{\Omega} \eta^2 h^2(x)({\mu^2}+|Du|^2)^{\frac{{2 q - p+\gamma}}{2}} \, dx,
\end{eqnarray}
where $\varepsilon>0$ will be chosen later.
Finally, similar arguments give 

\begin{eqnarray}\label{I6}
|{I}_6| &=&   \gamma \, \left |\int_{\Omega} \sum_{i,s=1}^n F_{\xi_i x_s}(x, Du) \eta^2  ({\mu^2} + |Du|^2)^{\frac{\gamma}{2}-1}  |Du| D_{x_i}(|Du|) u_{x_s} \, dx \right |\cr\cr
%&=&2 \gamma \, \int_{\Omega} \eta^2 (1 + |Du|^2)^{\gamma-1}  |Du| \sum_{i,s=1}^n f_{\xi_i x_s}(x, Du)  D_{x_i}(|Du|) u_{x_s} \, dx\\
&\le&  \, \gamma \, \int_{\Omega} \eta^2 ({\mu^2}+|Du|^2)^{\frac{\gamma-1}{2}}  \left |\sum_{i,s=1}^n F_{\xi_i x_s}(x, Du)  D_{x_i}(|Du|) u_{x_s} \right |\, dx  \cr\cr
%&\le&  \, \gamma \, \int_{\Omega} \eta^2  h(x)({\mu^2} + |Du|^2)^{\frac{{q-2+\gamma}}{2}}\, \left | \sum_{i,s=1}^n D_{x_i}(|Du|) u_{x_s} \right | \, dx\cr\cr
&\le&  \, \gamma \, \int_{\Omega} \, \eta^2 \, h(x)({\mu^2} + |Du|^2)^{\frac{{q-1+\gamma}}{2}} |D^2 u| \, dx\cr\cr
&\le& \varepsilon \int_{\Omega} \eta^2 |D^2u|^2 ({\mu^2} + |Du|^2)^{\frac{p-2+\gamma}{2}}   \, dx\\\nonumber
&&+ C_\varepsilon \gamma^2 \,  \, \int_{\Omega} \eta^2  h^2(x)({\mu^2} + |Du|^2)^{\frac{{2q - p+\gamma}}{2} }  \, dx.
\end{eqnarray}
Plugging  \eqref{I1}, \eqref{I4}, \eqref{I5} and \eqref{I6} in \eqref{ristart} we obtain
\begin{eqnarray*}
&& \int_{\Omega} \eta^2 ({\mu^2} +|Du|^2)^{\frac{\gamma}{2}}{\sum_{i,j,s=1}^n} F_{\xi_i \xi_j}(x, Du) u_{x_j x_s}  u_{x_s x_i} \, dx\cr\cr
&\le&  \frac{1}{2} \int_{\Omega} \eta^2  ({\mu^2} + |Du|^2)^{\frac{\gamma}{2}}  \, \sum_{i,j,s =1}^n F_{\xi_i \xi_j}(x, Du) u_{x_j x_s} u_{x_i x_s} \, dx\cr\cr
&&+2\varepsilon \int_{\Omega} \eta^2 ({\mu^2} + |Du|^2)^{\frac{p-2+\gamma}{2}}  |D^2 u|^2 \, dx \cr\cr
&&+ C_\varepsilon (1+\gamma^2)  \, \int_{\Omega} \eta^2  h^2(x)({\mu^2} + |Du|^2)^{\frac{{2 q - p +\gamma}}{2} } \, dx\cr\cr
&& + C\,  {\int_{\Omega} |D \eta|^2 \,  {({\mu}^2 + |Du|^2)^{\frac{p + \gamma}{2}}}\, dx+C \int_{\Omega}|D \eta|^2 \,\Big(1 +  (\mu^2 + |Du|^2)^{\frac{q+\gamma}{2}} \Big) \, dx.}%\cr\cr
%&&+C_\varepsilon \int_{\Omega} |D \eta|^2 \,  ({\mu^2} + |Du|^2)^{\frac{\tcb{p+\gamma}}{2}}  \, dx.
\end{eqnarray*}
Reabsorbing the first integral in the right hand side by the left hand side, we derive
\begin{eqnarray*}
&& \frac{1}{2}\int_{\Omega} {\sum_{i,j,s=1}^n} \eta^2 ({\mu^2} +|Du|^2)^{\frac{\gamma}{2}}  F_{\xi_i \xi_j}(x, Du) u_{x_j x_s} u_{x_s x_i} \, dx\cr\cr
 &\le &2\varepsilon \int_{\Omega} \eta^2 ({\mu^2} + |Du|^2)^{\frac{p-2+\gamma}{2}}  |D^2 u|^2 \, dx \cr\cr
&&+ C_\varepsilon (1+\gamma^2) \,  \int_{\Omega} \eta^2 \, {h^2(x)} ({\mu^2} + |Du|^2)^{\frac{{2 q - p +\gamma}}{2}}  \, dx\cr\cr
&&+C \int_{\Omega} |D \eta|^2 \,\big(1 +  (\mu^2 + |Du|^2)^{\frac{q+\gamma}{2}} \big) \, dx.
\end{eqnarray*}
Using assumption \eqref{(F2)(F3)} in the left hand side of previous estimate, we obtain
\begin{eqnarray*}
&& \frac{{\lambda}}{2}\int_{\Omega} \eta^2 ({\mu^2} + |Du|^2)^{\frac{p-2+\gamma}{2}}  |D^2 u|^2 \, dx \cr\cr
 &\le &2\varepsilon \int_{\Omega} \eta^2 ({\mu^2} + |Du|^2)^{\frac{p-2+\gamma}{2}}  |D^2 u|^2 \, dx \cr\cr
&&+ C_\varepsilon (1+\gamma^2) \, \int_{\Omega} \eta^2h^2(x)({\mu^2} + |Du|^2)^{\frac{{2 q - p +\gamma}}{2} }\, dx\cr\cr
&&+C \int_{\Omega} |D \eta|^2 \,\big(1 +  (\mu^2 + |Du|^2)^{\frac{q+\gamma}{2}} \big) \, dx.
\end{eqnarray*}
Choosing {for instance} ${\varepsilon=\frac{\lambda }{8}}$, we can reabsorb the first integral in the right hand side by the left hand side, thus getting
\begin{eqnarray*}
&& \int_{\Omega} \eta^2 ({\mu^2} + |Du|^2)^{\frac{p-2+\gamma}{2}}  |D^2 u|^2 \, dx \cr\cr
 &\le & C (1+\gamma^2)  \, \int_{\Omega} \eta^2 h^2(x)({\mu^2} + |Du|^2)^{\frac{{2 q - p +\gamma}}{2} } \, dx\cr\cr
&&+C \int_{\Omega}  |D \eta|^2 \,\big(1 +  (\mu^2 + |Du|^2)^{\frac{q+\gamma}{2}} \big) \, dx,
\end{eqnarray*}
i.e. the thesis. %\textcolor{teal}{controllare alla fine tutte le costanti}
\end{proof}

\medskip

\section{The a priori estimate}

In this section we shall prove a uniform a priori estimate for the Lipschitz norm of the minimizers of the functional at \eqref{functional} by the use of a Moser type iteration argument. More precisely, we are going to prove the following

\begin{theorem}\label{thmapriori}
Let $u\in W^{1,\infty}_{\loc}(\Omega)$ such that $(\mu^2+|Du|^2)^{\frac{p-2}{4}}Du\in W^{1,2}_{\loc}(\Omega; {\mathbb{R}^n)}$ be  a local minimizer of the functional at \eqref{functional}, under the assumptions \eqref{(F2)(F3)} and \eqref{(F4)}. Assume that there exists $r\ge n$ such that the exponents $1<p\le q$ appearing in \eqref{(F2)(F3)}, \eqref{(F4)} satisfy the following alternative conditions: either
   
    { i)} $p=q$ and
 $h(x)$ satisfies \eqref{F5F6}  with $r=n$ and $\alpha>4n$;\\ 
 %Then
 %$u\in W^{1,\infty}_{\mathrm{loc}}(\Omega)$ and the following estimate
%\begin{equation}
%\label{maincaso1}
%{||Du||_{L^{\infty}(B_{\rho})} \le \frac{C}{( R - \rho )^{\frac{n}{p}}}  \left(\int_{B_{R}} \{1 + |Du|^p \}
%\, dx\right)^{\frac{1}{p}}}
%\end{equation}
%{holds true for every $B_\rho\subset B_R\Subset\Omega,$ with $C \equiv C(n, p, \lambda, \Lambda, h);$}
 or
 
 { ii)}  $p<q$,  $\displaystyle \frac{q}{p} < 1+\frac{1}{n}-\frac{1}{r}$, and $h(x)$ satisfies \eqref{F5F6} with $r>n$ and $\alpha=0$; 
 %Then $u\in W^{1,\infty}_{\mathrm{loc}}(\Omega)$ and the following estimate
%{\begin{eqnarray}
%\label{maincaso2}
	%||Du||_{L^{\infty}(B_{\rho})} &\le&
%\frac{C}{(R-\rho)^{n \theta}}  \left(\int_{B_{R}} \{1 + |Du|^p \}
%\, dx\right)^{\theta}
%\end{eqnarray}}
%{holds true for every $B_\rho\subset B_R\Subset\Omega,$ with $C \equiv C(n, r, p, q, \lambda, \Lambda, \mathcal{L})$ and \[
%\theta := \frac{r-n}{p(r-n)-rn(q-p)};
%\]}
%\textcolor{magenta}{stavo cercando di confrontarlo con il risultato di Advances , anche se là è diverso perché c'era $m := \frac{r}{r-2}$ - ho fatto un po' di conti e sembra veramente diverso. Il nostro esponente sembra molto più elegante. D'altra parte ci solitamente ci si aspetta che $\theta > 1$ anche se negli Advances non è scritto esplicitamente, il che equivale a 
%\[
%\frac{q}{p} > 1 + \left (1 - \frac{1}{p} \right ) \left (\frac{1}{n} - \frac{1}{r}\right )
%\]
%che sembra ragionevole (perché se il nostro bound è al pelo, è ragionevole che se ci togli roba ci vai sotto, ma non mi pare scontato. Mentre scrivo, mi viene in mente la stima di Bella-Schaffner ma non mi pare ci si possa ricondurre a quello}
\\
or

{ iii)} $p<q$,  $\displaystyle \frac{q}{p} = 1+\frac{1}{n}-\frac{1}{r}$ and $h(x)$ satisfies \eqref{F5F6} with $r>n$
 and $\alpha>0.$
 \\
 Of course, these conditions are mutually exclusive.
 Under any of them,
 estimate  \eqref{mainaptotale} with the function $G(t)$ defined by
 \begin{equation}
\label{DefinizioneG}
G(t)=
\left \{
\begin{array}{lll} \!\!\!\!\!\!&t^{\frac1p}\qquad\qquad\qquad\qquad\qquad\qquad \qquad &\text{in the case { i)}}\\
    \!\!\!\!\!\! & t^\frac{r-n}{p(r-n)-rn(q-p)} \,\,\quad\qquad\qquad\qquad\qquad & \text{in the case { ii)}}\\
&\exp\Big(\frac{r}{2p}\cdot t^{\frac{2(r-n)}{\alpha n}}\Big)t^{\frac{r}{np}} . \qquad\quad & \text{in the case { iii)}},
\end{array}
\right.
\end{equation}
 holds true  for every $B_\rho\subset B_R\Subset\Omega,$ with $C \equiv C(n, r, p, q, \lambda, \Lambda, \alpha, h)$.
\end{theorem}
\begin{proof}
Let us fix a ball $B_{R_0}\Subset \Omega$ and radii $\rho_{0}<\tilde \rho_0<\rho<R<\tilde R_0<R_{0}$ that will be needed in the  iteration procedures, constituting the essential steps in our proof. We shall suppose without loss of generality that $R_0 \le 1$ and that inequalities between functions that hold asymptotically at infinity are valid for all values of the parameter, since our { reasoning remains valid by appropriately modifying the functions near zero}.

 Let us choose $\eta\in C^1_0(B_{{R}})$ such that $\eta=1$ on $B_{{\rho}}$ and $|D\eta|\le \frac{C}{{R-\rho}}$, so that \eqref{Caccioppoli} implies
 \begin{eqnarray}\label{primastima}
&& \int_{B_{{R}}} \eta^2 (\mu^2 + |Du|^2)^{\frac{p-2 + \gamma}{2}}  |D^2 u|^2 \, dx  \cr\cr 
&\le & C (1+\gamma^2) \, \int_{B_{{R}}} \eta^2h^2(x) (\mu^2 + |Du|^2)^{\frac{2q-p + \gamma}{2}}  \, dx \cr\cr
&&\quad+  \frac{C}{({R-\rho})^2}  \, \int_{B_{{R}}} \Big(1+  (\mu^2 + |Du|^2)^{\frac{q+\gamma}{2}} \Big) \, dx.
\end{eqnarray}
The Sobolev  inequality yields
\begin{eqnarray*}
&&\left ( \int_{B_{{R}}}  \eta^{2^*}\Big(1+(\mu^2 + |Du|^2)^{\frac{p + \gamma}{4}}\Big) ^{2^*}\, dx\right )^{\frac{2}{2^*}} \le \, C \, \int_{B_{{R}}} \Big|D \Big(\eta\big(1+(\mu^2 + |Du|^2)^{\frac{p + \gamma}{4}}\big)\Big)\Big|^2 \, dx 	\cr\cr
&\le&C(1+\gamma^2)\int_{B_{{R}}} \eta^2 (\mu^2 + |Du|^2)^{\frac{p-2 + \gamma}{2}}  |D^2 u|^2\,dx+C \int_{B_{{R}}} |D \eta|^2 \, \Big(1 +(\mu^2 + |Du|^2)^{\frac{p + \gamma}{2}}\Big)  \, dx,
\end{eqnarray*}
where, as usual, we set
$$ 2^*=\begin{cases}
	\displaystyle{\frac{2n}{n-2}}\qquad\qquad\qquad\qquad\qquad \text{if}\,\,n\ge3\cr\cr
	\text{any finite exponent}\, \qquad\quad\qquad \text{if}\,\,n=2.
	\end{cases}$$
	Using estimate \eqref{primastima} to control the first integral in the right hand side of previous inequality, we obtain
\begin{eqnarray}\label{stimapreite}
&&\left ( \int_{B_{{R}}}  \eta^{2^*}\Big(1 +(\mu^2 + |Du|^2)^{\frac{p + \gamma}{2}}\Big)^{ \frac{2^*}{2}}\, dx\right )^{\frac{2}{2^*}} 	\cr\cr
&\le& C (1+\gamma^2)^2 \int_{B_{{R}}} \eta^2 h^2(x) (\mu^2 + |Du|^2)^{\frac{2q-p + \gamma}{2}}  \, dx \cr\cr
&&+\frac{C(1+\gamma^2)}{({R - \rho})^{2}} \int_{B_{{R}}}  \, \big(1 +  (\mu^2 + |Du|^2)^{\frac{p+\gamma}{2}} \big) \, dx\cr\cr
&&+\frac{C(1+\gamma^2)}{({R - \rho})^{2}} \int_{B_{{R}}}  \,  \big(1 +  (\mu^2 + |Du|^2)^{\frac{q+\gamma}{2}} \big) \, dx\cr\cr
&\le& C (1+\gamma^2)^2 \int_{B_{{R}}} \eta^2 h^2(x) (\mu^2 + |Du|^2)^{\frac{2q-p + \gamma}{2}}  \, dx \cr\cr
&&+\frac{C(1+\gamma^2)}{({R - \rho})^{2}} \int_{B_{{R}}}  \,  \big(1 +  (\mu^2 + |Du|^2)^{\frac{q+\gamma}{2}} \big) \, dx.
\end{eqnarray}
From now on, to simplify the presentation, we shall use the following notations
\begin{equation}\label{defLV}
 V(Du):= \max\{1,(\mu^2+|Du|^2)^{\frac{1}{2}}\}. \qquad \qquad L(h):= h^{r-n}\log^{\alpha}(e+h).   
\end{equation}
In this way, we can write  \eqref{stimapreite} as follows
\begin{eqnarray}\label{stimapreitebis}
&&\left ( \int_{B_{{R}}}  \eta^{2^*}V(Du)^{(p+\gamma)\frac{2^*}{2}}\, dx\right )^{\frac{2}{2^*}} 	\cr\cr
&\le& C (1+\gamma^2)^2  \int_{B_{{R}}}  \eta^2h^2 V(Du)^{2q-p + \gamma}  \, dx\cr\cr
&&\quad+\,\frac{2C(1+\gamma^2)}{(R - \rho)^{2}} \int_{B_{{R}}}  \,  V(Du)^{q + \gamma}  \, dx \cr\cr
&\le &  C (1+\gamma^2)^2  \int_{{\{x: h(x) >\varepsilon\}}}\eta^2
\frac{1}{{\big[L(h)\big]^{\frac{2}{n}}}} \big[h^n L(h)\big]^{\frac{2}{n}} V(Du)^{2q-p + \gamma}  \, dx \cr \cr 
&& \quad+ \,C\varepsilon^2 (1+\gamma^2)^2  \int_{{\{x: h(x) < \varepsilon\}}}\eta^2 V(Du)^{2q-p + \gamma}   \, dx\cr\cr
&&\quad+\frac{C(1+\gamma^2)}{(R - \rho)^{2}} \int_{B_{{R}}}  \,  V(Du)^{q + \gamma}  \, dx \cr \cr 
&\le & C\frac{(1+\gamma^2)^2}{{L^{\frac{2}{n}}(\varepsilon)}} \|V(Du)\|_{L^\infty(B_{R})}^{2(q-p)} \int_{B_{{R}}}\eta^2 {[h^n L(h)]^{\frac{2}{n}}} V(Du)^{p + \gamma} \,dx\cr\cr
&& \quad+ \,C\varepsilon^2(1+\gamma^2)^2 \|V(Du)\|_{L^\infty(B_{{R}})}^{2(q-p)} \int_{B_{{R}}} V(Du)^{p + \gamma} \, dx\cr\cr
&&\quad+\,\frac{C(1+\gamma^2)}{({R - \rho})^{2}} \|V(Du)\|_{L^\infty(B_{{R}})}^{q-p}  \int_{B_{{R}}}  \,  V(Du)^{p + \gamma}  \, dx\cr \cr
&\le & C\frac{(1+\gamma^2)^2}{{L^{\frac{2}{n}}(\varepsilon)}} \|V(Du)\|_{L^\infty(B_{R})}^{2(q-p)} \int_{B_{{R}}}\eta^2 {[h^n L(h)]^{\frac{2}{n}}} V(Du)^{p + \gamma} \,dx\cr\cr
&& \quad+ \,C\varepsilon^2(1+\gamma^2)^2 \|V(Du)\|_{L^\infty(B_{{R}})}^{2(q-p)} \int_{B_{{R}}} V(Du)^{p + \gamma} \, dx\cr\cr
&&\quad+\,\frac{C(1+\gamma^2)}{({R - \rho})^{2}} \|V(Du)\|_{L^\infty(B_{{R}})}^{q-p}  \int_{B_{{R}}}  \, V(Du)^{p + \gamma}  \, dx\cr \cr
&=:& I + II + III,
\end{eqnarray}
%\textcolor{teal}{secondo me al posto del termine che ho messo in verde ci vorrebbe $1 + \|V(Du)\|_{L^\infty(B_{\bar R})}^{q-p}$ Se concordi, cambio ovunque} \textcolor{magenta}    {concordo!}
where, in the {second to last} inequality, we used the a priori assumption $Du\in L^\infty_{\mathrm{loc}}(\Omega)$ and that the function $L(t)$ is strictly increasing.
\\
In what follows, we will make explicit the computations in the case $n > 2;$ the case $n = 2$ can be treated with similar methods as in {\cite{CGHPdN20, EPdN23, EPdN25}}.

We estimate the integral $I$ by using assumption \eqref{F5F6} as follows
\begin{eqnarray*}
I &\le& C\frac{(1+\gamma^2)^2}{{L^{\frac{2}{n}}(\varepsilon)}} \|V(Du)\|_{L^\infty\left(B_{{R}}\right)}^{2(q-p)} \left (\int_{B_{{R}}}h^n L(h) \, dx \right )^{2/n}\cr\cr
&&\qquad\qquad\cdot\left (\int_{B_{{R}}}\eta^{2^*}V(Du)^{(p + \gamma)\frac{ n}{n-2} } \, dx \right )^{\frac{n-2}{n}}.
\end{eqnarray*}
Setting
\[
H^2 := C \, \left (\int_{B_{R_0}}  h^n L(h) \, dx \right )^{2/n} = C \, \left (\int_{B_{R_0}}  h^r\log^\alpha(e+h) \, dx \right )^{2/n},
\]
 we have that
\begin{eqnarray*}
I &\le& \frac{H^2(1+\gamma^2)^2}{{L^{\frac{2}{n}}(\varepsilon)}} \|V(Du)\|_{L^\infty\left(B_{{R}}\right)} ^{2(q-p)}\left (\int_{B_{R}}\eta^{2^*} V(Du)^{(p + \gamma)\frac{ n}{n-2} } \, dx \right )^{\frac{n-2}{n}}.
\end{eqnarray*}
 At this point we choose $\varepsilon$ so that
\[
\frac{H^2(1+\gamma^2)^2}{{L^{\frac{2}{n}}(\varepsilon)}} \|V(Du)\|_{L^\infty\left(B_{{R}}\right)}^{2(q-p)} = \frac{1}{4}\quad\Longleftrightarrow  \quad L(\varepsilon)=2^{n}H^n(1+\gamma^2)^n\|V(Du)\|_{L^\infty\left(B_{{R}}\right)}^{n(q-p)}
\]
then, by the strict monotonicity of the function $L$,
 we have 
\begin{eqnarray*}
 \varepsilon = {{L}}^{-1} \Big(2^{n} H^n (1+\gamma^2)^n \|V(Du)\|_{L^\infty\left(B_{{R}}\right)}^{n(q-p)}\Big) 
\end{eqnarray*}
and, with this choice, estimate \eqref{stimapreitebis} becomes
\begin{eqnarray*}
&&\left (\int_{B_{{R}}} \eta^{2^*} V(Du)^{(p + \gamma)\frac{ n}{n-2} } \,dx \right )^{\frac{n-2}{n}}\le \frac{1}{4} \left (\int_{B_{{R}}}\eta^{2^*} V(Du)^{(p + \gamma)\frac{ n}{n-2}} \, dx \right )^{\frac{n-2}{n}}\cr\cr 
&&\qquad +\frac{C(1+\gamma^2)}{({R - \rho})^{2}}  \|V(Du)\|_{L^\infty(B_{{R}})}^{q-p} \int_{B_{R}}  \,  V(Du)^{p + \gamma}  \, dx\cr\cr
&&\,\, +C\left[(1+\gamma^2)\|V(Du)\|_{L^\infty\left(B_{{R}}\right)}^{q-p} {{L}}^{-1} \Big(2^{n} H^n (1+\gamma^2)^n \|V(Du)\|_{L^\infty\left(B_{{R}}\right)}^{n(q-p)}\Big) 
\right]^2\cr\cr 
&&\qquad\qquad\cdot\int_{B_{{R}}} V(Du)^{p + \gamma} \, dx\cr\cr 
&\le& \frac{1}{4} \left (\int_{B_{{R}}}\eta^{2^*} V(Du)^{(p + \gamma)\frac{n}{n-2}} \, dx \right )^{\frac{n-2}{n}} +\frac{C(1+\gamma^2)}{({R - \rho})^{2}}  \|V(Du)\|_{L^\infty(B_{{R}})}^{q-p}   \int_{B_{{R}}}  \,  V(Du)^{p + \gamma}  \, dx\cr\cr
&+&C\left[(1+\gamma^4) \|V(Du)\|_{L^\infty(B_{{R}})}^{2(q-p)} {{L}}^{-1} \Big((1+4^{n} H^{2n}) {(1+\gamma^2)^{2n}}\Big) {{L}}^{-1} \Big(\|V(Du)\|^{2n(q-p)}_{L^\infty(B_{{R}})}\Big) 
\right] \cr\cr
&& \qquad\qquad\qquad\cdot\int_{B_{{R}}} V(Du)^{p + \gamma} \, dx 
\end{eqnarray*}
where, in the last line, we used Lemma \ref{elem1} in the case $r>n$, since by the definition of $L(t)$ we have that $L^{-1}(\tau)$ is equivalent to the function $\tau^{\frac{1}{r-n}}\log^{-\frac{\alpha}{r-n}}(\tau)$. If $r=n$, $L^{-1}(\tau)=\exp(\tau^{\frac{1}{\alpha}})-e$ and an easy application of Young's inequality yields also in this case that $L^{-1}(st)\le  \sqrt{L^{-1}(s^2)L^{-1}(t^2)}+e $.

Reabsorbing the first integral in the right hand side by the left hand side {of the previous estimate}, setting
$$\Theta= {1 + 4^nH^{2n}}$$
and recalling that $\eta=1$ on $B_\rho$, we finally have
\begin{eqnarray}\label{baseiter}
&&\left (\int_{B_{\rho}} V(Du)^{(p + \gamma)\frac{n}{n-2}} \,dx \right )^{\frac{n-2}{n}} \cr\cr
&\le& C\left[(1+\gamma^4) \|V(Du)\|_{L^\infty(B_{\tilde R_0})}^{2(q-p)} {{L}}^{-1} \Big(\Theta (1+\gamma^2)^{2n}\Big){{L}}^{-1} \Big(\|V(Du)\|^{2n(q-p)}_{L^\infty(B_{ \tilde R_0})}\Big) 
\right]\cr\cr
&&\qquad\qquad\qquad\cdot \int_{B_R} V(Du)^{p + \gamma} \, dx\cr\cr
&&\qquad+\frac{C(1+\gamma^2)}{(R - \rho)^{2}}   \|V(Du)\|_{L^\infty(B_{\tilde R_0})}^{q-p}  \int_{B_{R}}  \,  V(Du)^{p + \gamma}  \, dx
\cr\cr
&\le& {C\Bigg[(1+\gamma^4) \|V(Du)\|_{L^\infty(B_{\tilde R_0})}^{2(q-p)} {{L}}^{-1} \Big(\Theta (1+\gamma^2)^{2n}\Big){{L}}^{-1} \Big(\|V(Du)\|^{2n(q-p)}_{L^\infty(B_{ \tilde R_0})}\Big) }\cr\cr
&&\qquad\qquad\qquad\qquad {+\frac{1+\gamma^2}{(R - \rho)^{2}}   \|V(Du)\|_{L^\infty(B_{\tilde R_0})}^{2(q-p)}  
\Bigg]\cdot\int_{B_{R}}  \, V(Du)^{p + \gamma}  \, dx.}
\end{eqnarray}
%{where we used that $R-\rho<1$ and that $1+x\le 2(1+x^2)$ and that $L^{-1}({2}\Theta (1+\gamma^4)^n)\ge L^{-1}({2}\Theta)$.}%\colorbox{yellow}{l' ultima l' ho scritta per noi ma è banale}
\begin{comment}
     where, in the last line, we used the fact that, by  \eqref{maggiore} and by virtue of the strict monotonicity of $L$, it holds the following $${{L}}^{-1} \Big(\Theta (1+\gamma^2)^{2n}\Big) {{L}}^{-1} \Big(\|V(Du)\|^{2n(q-p)}_{L^\infty(B_{ R})}\Big)\|V(Du)\|_{L^\infty(B_{ R})}^{q-p}\ge C(K, \Theta).$$
 Indeed,
 $$L^{-1} \Big(\Theta (1+\gamma^2)^{2n}\Big) {{L}}^{-1} \Big(\|V(Du)\|^{2n(q-p)}_{L^\infty(B_{ R})}\Big)\|V(Du)\|_{L^\infty(B_{ R})}^{q-p}$$
 $$\ge L^{-1} (\Theta ) {{L}}^{-1} (K^{2n(q-p)})K^{q-p}=:C(\Theta,K)$$
 and this is equivalent to
 $$\|V(Du)\|_{L^\infty(B_{ R})}^{q-p}\le C {{L}}^{-1} \Big(\Theta (1+\gamma^2)^{2n}\Big) {{L}}^{-1} \Big(\|V(Du)\|^{2n(q-p)}_{L^\infty(B_{ R})}\Big)\|V(Du)\|_{L^\infty(B_{ R})}^{2(q-p)}$$
 
% To shorten the notation, we set
% \begin{equation}
%     \label{legamePsigPhi}
% \textcolor{red}{\Psi(s^2) := s^2 \Phi^{-1} \big(s^2\big)}
% \end{equation}
% so that we can rewrite \eqref{baseiter} as follows
% \begin{eqnarray}\label{baseiter2}
% &&\left (\int_{B_{\rho}} V(Du)^{\frac{(p + \gamma) n}{n-2}} \,dx \right )^{\frac{n-2}{n}} \cr\cr
% &\le& \frac{1}{H} \Psi \Big(\sqrt{CH}(1+\gamma^2) g(\|Du\|_{\infty}^2)\Big) \int_{B_R} V(Du)^{p + \gamma} \, dx\cr\cr
% &&\qquad+\frac{C}{(R - \rho)^{2}} \int_{B_{R}}  \,  V(|Du|)^{p + \gamma}  \, dx 
% \end{eqnarray}
\end{comment}

With radii $\tilde \rho_0\le  \rho< R\le \tilde R_0$ fixed at the beginning of the section, we define the decreasing sequence of radii by setting
$$\rho_i= \tilde\rho_0+\frac{ \tilde R_0- \tilde\rho_0}{2^i}.$$
Let us also define the following increasing sequence of exponents
$$p_0=p \qquad\qquad {p_{i}}={p_{i-1}}\frac{2^*}{2}=p_0\left(\frac{2^*}{2}\right)^{i}.$$
By virtue the a priori assumption $u\in W^{1,\infty}_{\mathrm{loc}}(\Omega)$, estimate \eqref{baseiter} holds true for $\gamma=0$ {which represents the first step of the iteration} and for every $ \tilde\rho_0<\rho<R< \tilde R_0$. Hence, we may iterate it  on the concentric balls $B_{\rho_i}$  {with $\gamma_i = p_i - p$}, thus obtaining
\begin{eqnarray}\label{stimapreitepassoi}
&&\left ( \int_{B_{\rho_{i+1}}}  V(Du)^{p_{i+1}}\, dx\right )^{\frac{1}{p_{i+1}}} 	\cr\cr
&\le& \displaystyle{\prod_{j=0}^{i}} \Bigg( Cp_j^4{{L}}^{-1} \Big(\Theta p_j^{4n}\Big) \|V(Du)\|_{L^\infty(B_{\tilde R_0})}^{2(q-p)}{{L}}^{-1} \Big(\|V(Du)\|^{2n(q-p)}_{L^\infty(B_{ \tilde R_0})}\Big)\cr\cr 
&&\qquad\qquad +\frac{Cp_j^2}{(\rho_j - \rho_{j+1})^{2}} \|V(Du)\|_{L^\infty(B_{\tilde R_0})}^{2(q-p)} \Bigg)^{\frac{1}{p_j}}\cdot\left(\int_{B_{\tilde R_0}}V(Du)^{p}  \,
dx\right)^{\frac{1}{p}}\cr\cr
&=& \displaystyle{\prod_{j=0}^{i}}\Bigg( Cp_j^4{{L}}^{-1} \Big(\Theta p_j^{4n}\Big) \|V(Du)\|_{L^\infty(B_{\tilde R_0})}^{2(q-p)}{{L}}^{-1} \Big(\|V(Du)\|^{2n(q-p)}_{L^\infty(B_{ \tilde R_0})}\Big)\cr\cr 
&&\qquad\qquad +\frac{C4^{j+1}p_j^2}{(\tilde R_0 - \tilde \rho_0)^{2}} \|V(Du)\|_{L^\infty(B_{\tilde R_0})}^{2(q-p)} \Bigg)^{\frac{1}{p_j}}\cdot\left(\int_{B_{\tilde R_0}}V(Du)^{p}  \,
dx\right)^{\frac{1}{p}}\cr\cr
&\le & \displaystyle{\prod_{j=0}^{i}}\Bigg[ C\left( \|V(Du)\|_{L^\infty(B_{\tilde R_0})}^{2(q-p)}{{L}}^{-1} \Big(\|V(Du)\|^{2n(q-p)}_{L^\infty(B_{ \tilde R_0})}\Big)+\frac{ \|V(Du)\|_{L^\infty(B_{\tilde R_0})}^{2(q-p)} }{(\tilde R_0 - \tilde \rho_0)^{2}}\right)\cr\cr 
&&\qquad\qquad\qquad\cdot\Big(p_j^4{{L}}^{-1} \Big(\Theta p_j^{4n}\Big)+4^{j+1}p_j^2\Big)\Bigg]^{\frac{1}{p_j}}\cdot\left(\int_{B_{\tilde R_0}}V(Du)^{p}  \,
dx\right)^{\frac{1}{p}}.
\end{eqnarray}
Note now that
\begin{eqnarray*}
&& \displaystyle{\prod_{j=0}^{i}}\Bigg[ C\left( \|V(Du)\|_{L^\infty(B_{\tilde R_0})}^{2(q-p)}{{L}}^{-1} \Big(\|V(Du)\|^{2n(q-p)}_{L^\infty(B_{ \tilde R_0})}\Big)+\frac{ \|V(Du)\|_{L^\infty(B_{\tilde R_0})}^{2(q-p)} }{(\tilde R_0 - \tilde \rho_0)^{2}}\right)\cr\cr 
&&\qquad\qquad\qquad\cdot\Big(p_j^4{{L}}^{-1} \Big(\Theta p_j^{4n}\Big)+4^{j+1}p_j^2\Big)\Bigg]^{\frac{1}{p_j}}\cr\cr
&=&\exp\Bigg(\sum_{j=0}^i \frac{1}{p_j}\log\left( C \|V(Du)\|_{L^\infty(B_{\tilde R_0})}^{2(q-p)}\Bigg({{L}}^{-1} \Big(\|V(Du)\|^{2n(q-p)}_{L^\infty(B_{ \tilde R_0})}\Big)+\frac{1}{(\tilde R_0 - \tilde \rho_0)^{2}}\Bigg)\right)\cr\cr
&&\qquad+\sum_{j=0}^i \frac{1}{p_j}\log\Big[p_j^4{{L}}^{-1} \Big(\Theta p_j^{4n}\Big)+4^{j+1}p_j^2 \Big]\Bigg)\cr\cr
&\le& \exp\Bigg(\sum_{j=0}^i \frac{1}{p_j}\log\left( C \|V(Du)\|_{L^\infty(B_{\tilde R_0})}^{2(q-p)}\Bigg({{L}}^{-1} \Big(\|V(Du)\|^{2n(q-p)}_{L^\infty(B_{ \tilde R_0})}\Big)+\frac{1}{(\tilde R_0 - \tilde \rho_0)^{2}}\Bigg)\right)\cr\cr
&& \qquad+\sum_{j=0}^i \frac{1}{p_j}\log\Big[p_j^4 4^{j+1} \Big ({{L}}^{-1} \Big(\Theta p_j^{4n}\Big)+1 \Big )\Big]\Bigg) \cr\cr
%&\le& \tcr{\exp\Bigg(\sum_{j=0}^i \frac{1}{p_j}\log\left( C \|V(Du)\|_{L^\infty(B_{\tilde R_0})}^{2(q-p)}\Bigg({{L}}^{-1} \Big(\|V(Du)\|^{2n(q-p)}_{L^\infty(B_{ \tilde R_0})}\Big)+\frac{1}{(\tilde R_0 - \tilde \rho_0)^{2}}\Bigg)\right)}\cr\cr
%&&\tcr{\qquad\cdot\exp\Bigg(\sum_{j=0}^i \frac{1}{p_j}\log\Big[4^{j+2}p_j^4 \Big]\Bigg)\cdot\exp\Bigg(\sum_{j=0}^i \frac{1}{p_j}\log\Big[{{L}}^{-1} \Big(\Theta p_j^{4n}\Big) + 1 \Big]\Bigg)}\cr\cr
%\end{eqnarray*}
%Since
%\begin{eqnarray*}
%&& \log\Big[p_j^4 4^{j+1} \Big ({{L}}^{-1} \Big(\Theta p_j^{4n}\Big)+1 \Big )\Big]\\ 
%&=& \log\Big[p_j^4 4^{j+1} \Big ] + \log \Big ({{L}}^{-1} \Big(\Theta p_j^{4n}\Big)+1 \Big )\\
%&\le & \, \log\Big[p_j^4 4^{j+2} \Big ] + \log \Big (2{{L}}^{-1} \Big(\Theta p_j^{4n}\Big) \Big )  + \log 2 \\
%& = & \, \log\Big[p_j^4 4^{j+2} \Big ] + \log \Big ({{L}}^{-1} \Big(\Theta p_j^{4n}\Big) \Big )  + 2 \log 2, 
%\end{eqnarray*}
%\tcm{where in the second-to-last line we used the elementary inequality $\log(a+b)\le \log(2a)+\log(2b)$ which is valid as long as for us $a, b > 1.$}
%\colorbox{yellow}{Siamo sicure che ${{L}}^{-1} \Big(\Theta p_j^{4n}\Big) >1$? }
%\\
% we have at the end
%\begin{eqnarray*}
%&& \displaystyle{\prod_{j=0}^{i}}\Bigg[ C\left( \|V(Du)\|_{L^\infty(B_{\tilde R_0})}^{2(q-p)}{{L}}^{-1} \Big(\|V(Du)\|^{2n(q-p)}_{L^\infty(B_{ \tilde R_0})}\Big)+\frac{ \|V(Du)\|_{L^\infty(B_{\tilde R_0})}^{2(q-p)} }{(\tilde R_0 - \tilde \rho_0)^{2}}\right)\cr\cr 
%&&\qquad\qquad\qquad\cdot\Big(p_j^4{{L}}^{-1} \Big(\Theta p_j^{4n}\Big)+4^{j+1}p_j^2\Big)\Bigg]^{\frac{1}{p_j}}\cr\cr
&\le& \exp\Bigg(\sum_{j=0}^i \frac{1}{p_j}\log\left( C \|V(Du)\|_{L^\infty(B_{\tilde R_0})}^{2(q-p)}\Bigg({{L}}^{-1} \Big(\|V(Du)\|^{2n(q-p)}_{L^\infty(B_{ \tilde R_0})}\Big)+\frac{1}{(\tilde R_0 - \tilde \rho_0)^{2}}\Bigg)\right)\cr\cr
&&\cdot\exp\Bigg(\sum_{j=0}^i \frac{1}{p_j}\log\Big[4^{j+2}p_j^4 \Big]\Bigg)\cdot\exp\Bigg(\sum_{j=0}^i \frac{1}{p_j}\log\Big (1+{{L}}^{-1} \Big(\Theta p_j^{4n}\Big) \Big ) \Bigg)\cr\cr 
&=:& A_i\cdot B_i\cdot C_i.
\end{eqnarray*}
One can easily check that
{\begin{eqnarray*}
   A_i&\le& \exp\Bigg(\log\left( C \|V(Du)\|_{L^\infty(B_{\tilde R_0})}^{2(q-p)}\Bigg({{L}}^{-1} \Big(\|V(Du)\|^{2n(q-p)}_{L^\infty(B_{ \tilde R_0})}\Big)+\frac{1}{(\tilde R_0 - \tilde \rho_0)^{2}}\Bigg)\right)\sum_{j=0}^\infty \frac{1}{p_j}\Bigg)\cr\cr
   &=& \exp\Bigg(\frac{n}{2p}\log\left( C \|V(Du)\|_{L^\infty(B_{\tilde R_0})}^{2(q-p)}\Bigg({{L}}^{-1} \Big(\|V(Du)\|^{2n(q-p)}_{L^\infty(B_{ \tilde R_0})}\Big)+\frac{1}{(\tilde R_0 - \tilde \rho_0)^{2}}\Bigg)\right)\Bigg)\cr\cr
   &=& \exp\Bigg(\log\left[ C \|V(Du)\|_{L^\infty(B_{\tilde R_0})}^{2(q-p)}\Bigg({{L}}^{-1} \Big(\|V(Du)\|^{2n(q-p)}_{L^\infty(B_{ \tilde R_0})}\Big)+\frac{1}{(\tilde R_0 - \tilde \rho_0)^{2}}\Bigg)\right]^\frac{n}{2p}\Bigg)\cr\cr
   &\le& \tilde C \|V(Du)\|^{\frac{n(q-p)}{p}}_{L^\infty\left(B_{ \tilde R_0}\right)} \Big({L}^{-1} \Big(\|V(Du)\|^{2n(q-p)}_{L^\infty\left(B_{ \tilde R_0}\right)}\Big)\Big)^{\frac{n}{2p}}+\tilde C\frac{ \|V(Du)\|^{\frac{n(q-p)}{p}}_{L^\infty\left(B_{ \tilde R_0}\right)} }{(\tilde R_0 - \tilde \rho_0)^{\frac{n}{p}}},
\end{eqnarray*}
where we used that $\sum_{j=0}^{{\infty}}\frac{1}{p_j}= \frac{n}{2 p}$, as well as 
\begin{eqnarray*}
B_i&\le & \exp\Bigg(\log 4\sum_{j=0}^\infty \frac{j+2}{p_j}+4\sum_{j=0}^\infty \frac{\log p_j}{p_j} \Bigg) \le C(n,p). 
\end{eqnarray*}}
\begin{comment}
The same arguments yield
\begin{eqnarray*}
&&\displaystyle{\prod_{j=0}^{i}}\Bigg(\frac{C2^{j+1}p_j^2}{(\tilde R_0- \tilde \rho_0)^{2}}\Big(1+\|V(Du)\|^{{q-p}}_{L^\infty\left(B_{ \tilde R_0}\right)}\Big) \Bigg)^{\frac{1}{p_j}} \\
&= & \exp\left(\sum_{j=0}^i \frac{1}{p_j}\log\left(\frac{C\Big(1+\|V(Du)\|^{{q-p}}_{L^\infty\left(B_{ \tilde R_0}\right)}\Big)}{( \tilde R_0- \tilde \rho_0 )^{2}}2^{j+1} p_j^{2}\right)\right)\\
&\le & \exp\left(\frac{C\Big(1+\|V(Du)\|^{{q-p}}_{L^\infty\left(B_{ \tilde R_0}\right)}\Big)}{( \tilde R_0- \tilde \rho_0 )^{2}}\sum_{j=0}^{+\infty} \frac{1}{p_j}+\log2\sum_{j=0}^{+\infty}\frac{j+1}{p_j}+2\sum_{j=0}^{+\infty} \frac{\log p_j}{p_j}\right)\\
&\le & \frac{{C(\Theta)\Big(1+\|V(Du)\|^{{q-p}}_{L^\infty\left(B_{ \tilde R_0}\right)}\Big)^{\frac{n}{2p}}}}{( \tilde R_0- \tilde \rho_0
)^{\frac{n}{p}}}.
\end{eqnarray*}
\end{comment}
%Now, we observe that
%$$\displaystyle{\prod_{j=0}^{i}}\left({{L}^{-1} \Big({2}\Theta p_j^{4n}\Big)} \right)^{\frac{1}{p_j}}=\exp\left(\sum_{j=0}^i \frac{1}{p_j}\log{{L}^{-1} \Big({2} \Theta p_j^{4n}\Big)}\right). $$
To bound the term $C_i$, we distinguish between the case $r=n$ and $r>n$ and we shall use the following {direct consequences of the definition of the function $L(t)$ at \eqref{defLV}} 
\begin{equation}\label{F6}
    \begin{cases}
L^{-1}(\tau)\le \exp\left[\left(\tau\right)^{{\frac{1}{\alpha }}}\right]\qquad\qquad\qquad \quad \text{if}\,\,\alpha>0\,\, \text{and}\,\, r=n\cr\cr
L^{-1}(\tau)\le \frac{\tau^{\frac{1}{r-n}}}{\log^{^{{\frac{\alpha }{r-n}}}}\big(e+L^{-1}(\tau)\big)} \qquad\qquad\,\,\, \text{if}\,\,\alpha\ge 0\,\, \text{and}\,\, r>n,
\end{cases}\end{equation}
 for all $\tau \ge \log^\alpha(e+1)$.
If $r=n,$ which, of course entails $p = q$, we use the first estimate in \eqref{F6}  to deduce that
$$C_i  =\exp\left(\sum_{j=0}^i \frac{1}{p_j}\log\Big(1+{{L}^{-1} \Big( \Theta p_j^{4n}\Big)}\Big)\right)
\le \exp\left(\sum_{j=0}^\infty \frac{\log\Big(1+\exp\Big(\Theta^{{\frac{1}{\alpha }}}p_j^{{\frac{4n}{\alpha }}}\Big)\Big)}{p_j}\right)$$ $$\le \exp\left(\Theta^{{\frac{1}{\alpha }}}\sum_{j=0}^\infty \frac{p_j^{{\frac{4n}{\alpha }}}}{p_j}\right) \exp\left( \sum_{j=0}^\infty \frac{\log 2}{p_j}\right)\le C(\Theta,n,p,\alpha), $$
since, by virtue of assumption i), we have ${\alpha>4n}$. If $r>n$, we use the second estimate in \eqref{F6} {(where we notice that the denominator in the right hand side is greater than $1$)}, to deduce that
{
\begin{eqnarray*}
    && C_i =\exp\left(\sum_{j=0}^i \frac{1}{p_j}\log\Big(1+{{L}^{-1} \Big( \Theta p_j^{4n}\Big)}\Big)\right)\le\exp\left(\sum_{j=0}^\infty \frac{1}{p_j}\left [\log\Big(1+ \Big(\Theta \, p_j^{4n}\Big)^{\frac{1}{r - n}}\Big) \right ] \right)\\
    &\le& \exp\left(\sum_{j=0}^\infty \frac{1}{p_j}\left [\log\Big( \Big(\Theta \, p_j^{4n}\Big)^{\frac{1}{r - n}}\Big) \right ] \right)\exp\left( \sum_{j=0}^\infty \frac{\log 2}{p_j}\right) \\
    &\le & \exp\left( \frac{1}{r - n}\sum_{j=0}^{+ \infty}  \frac{\log \Theta}{p_j} + \frac{4n}{r-n}  \sum_{j=0}^{+ \infty}    
    \frac{\log p_j}{p_j} \right )\exp\left(\sum_{j=0}^i \frac{\log 2}{p_j}\right)\le C(\Theta,n,p,r),
\end{eqnarray*}
}{where we used the elementary inequality $\log (a + b) \le \, \log (2a) + \log(2b),$ being $a, b \ge \, 1.$} {Hence, collecting the estimates of $A_i,B_i,C_i$, we deduce that
\begin{eqnarray*}
&& \displaystyle{\prod_{j=0}^{i}}{\Bigg[ C\left( \|V(Du)\|_{L^\infty(B_{\tilde R_0})}^{2(q-p)}{{L}}^{-1} \Big(\|V(Du)\|^{2n(q-p)}_{L^\infty(B_{ \tilde R_0})}\Big)+\frac{ \|V(Du)\|_{L^\infty(B_{\tilde R_0})}^{2(q-p)} }{(\tilde R_0 - \tilde \rho_0)^{2}}\right)}\cr\cr 
&&\qquad\qquad\qquad\cdot\Big({p_j^4{{L}}^{-1} \Big(\Theta p_j^{4n}\Big)+4^{j+1}p_j^2}\Big)\Bigg]^{\frac{1}{p_j}}\cr\cr
&\le&  C \|V(Du)\|^{\frac{n(q-p)}{p}}_{L^\infty\left(B_{ \tilde R_0}\right)} \Big({L}^{-1} \Big(\|V(Du)\|^{2n(q-p)}_{L^\infty\left(B_{ \tilde R_0}\right)}\Big)\Big)^{\frac{n}{2p}}+C\frac{  \|V(Du)\|^{\frac{n(q-p)}{p}}_{L^\infty\left(B_{ \tilde R_0}\right)} }{(\tilde R_0 - \tilde \rho_0)^{\frac{n}{p}}}
\end{eqnarray*}
with $C=C(\Theta,n,p,q,r,\alpha)$.}
Therefore, we can let $i\to \infty$ in \eqref{stimapreitepassoi} thus getting
\begin{eqnarray}\label{stimaquasifin}
	&&||V(Du)||_{L^{\infty}(B_{\tilde\rho_0})} \le { C \|V(Du)\|^{\frac{n(q-p)}{p}}_{L^\infty\left(B_{\tilde  R_0}\right)} \Big({L}^{-1} \Big(\|V(Du)\|^{2n(q-p)}_{L^\infty\left(B_{ \tilde R_0}\right)}\Big)\Big)^{\frac{n}{2p}}} \cr\cr 
    &&\qquad\qquad\cdot \left(\int_{B_{  R_0}}V(Du)^{p}
\, dx\right)^{\frac{1}{p}}+{C\frac{  \|V(Du)\|^{\frac{n(q-p)}{p}}_{L^\infty\left(B_{\tilde  R_0}\right)} }{(\tilde R_0 - \tilde \rho_0)^{\frac{n}{p}}} } 
    \cdot \left(\int_{B_{  R_0}}V(Du)^{p}
\, dx\right)^{\frac{1}{p}}
%\cr\cr
%&\le& \frac{C}{( \tilde R_0- \tilde \rho_0 )^{\frac{n}{p}}}\left(1+\|V(Du)\|^{\frac{n(q-p)}{p}}_{L^\infty\left(B_{ \tilde R_0}\right)}\left({L}^{-1} \Big(\|V(Du)\|^{2n(q-p)}_{L^\infty\left(B_{ \tilde R_0}\right)}\Big)\right)^{\frac{n}{2p}}\right)\cr\cr  
    %&\cdot& \left(\int_{B_{ \tilde R_0}}V(Du)^{p}
%\, dx\right)^{\frac{1}{p}},
\end{eqnarray}
with $C=C(\Theta,n,p,{q},r,\alpha)$. 
\begin{comment}and where we used that \begin{equation}\label{eq-elem}
 (1+x)\big(1+L^{-1}(x^{\frac{2}{p}})\big)^{\frac{n}{2p}}\le C\left(1+x\big(L^{-1}(x^{\frac{2}{p}})\big)^{\frac{n}{2p}}\right).   
\end{equation}
Indeed, if $x^{\frac{2}{p}}\le \log^{{\alpha}}(e+1)$ then the strict monotonicity of the function $L(t)$ implies  $L^{-1}(x^{\frac{2}{p}})\le L^{-1}(\log^{{\alpha}}(e+1))= {L^{-1} (L(1)) = 1}$ and therefore the right hand side of \eqref{eq-elem} can be bounded with a constant. On the other hand, if $x^{\frac{2}{p}}> \log^{{\alpha}}(e+1)$, we have $$1+x\le \log^{\frac{{{\alpha}}p}{2}}(e+1)+x\le 2x$$ and $L^{-1}(x^{\frac{2}{p}})>1$ that yields $$\big(1+L^{-1}(x^{\frac{2}{p}})\big)^{\frac{n}{2p}}\le \big(2L^{-1}(x^{\frac{2}{p}})\big)^{\frac{n}{2p}}.$$ Multiplying the last two inequalities we get
$$ (1+x)\big(1+L^{-1}(x^{\frac{2}{p}})\big)^{\frac{n}{2p}}\le Cx\big(L^{-1}(x^{\frac{2}{p}})\big)^{\frac{n}{2p}}$$
and summing the results we end up with \eqref{eq-elem}.
\end{comment}
Note that,  in case $q=p$,  estimate \eqref{stimaquasifin} reads as
\begin{equation*}\label{stimafinp=q}
	||V(Du)||_{L^{\infty}(B_{\tilde\rho_0})} \le C\left(\int_{B_{  R_0}}V(Du)^{p}
\, dx\right)^{\frac{1}{p}}+\frac{C}{({\tilde{R_0}- \tilde{\rho_0}})^{\frac{n}{p}}}  \left(\int_{B_{  R_0}}V(Du)^{p}
\, dx\right)^{\frac{1}{p}}
\end{equation*}
and this ends the proof in this case. Remark also that {the restriction $\alpha > 4n$ it is needed only here, when using} assumption {i)}.
\\
In case $p<q$, by the assumption at ii) we have  that $r>n$ and so we may use the second estimate in \eqref{F6} in the right hand side of   \eqref{stimaquasifin} thus getting
\begin{eqnarray}\label{stimaquasifinbis}
	||V(Du)||_{L^\infty\left(B_{ \tilde\rho_0}\right)} 
&\le &  {\frac{\|V(Du)\|^{\frac{q-p}{p}\frac{rn}{r-n}}_{L^\infty\left(B_{\tilde  R_0}\right)}}{\log^{{\frac{\alpha n}{2p(r-n)}}}\left(e+L^{-1} \Big(\|V(Du)\|^{2n(q-p)}_{L^\infty\left(B_{ \tilde R_0}\right)}\Big)\right)}} \cdot\left(\int_{B_{  R_0}}V(Du)^{p}
\, dx\right)^{\frac{1}{p}}\cr\cr  
    &+& {C\frac{  \|V(Du)\|^{\frac{n(q-p)}{p}}_{L^\infty\left(B_{ \tilde R_0}\right)} }{(\tilde R_0 - \tilde \rho_0)^{\frac{n}{p}}} } \cdot\left(\int_{B_{  R_0}}V(Du)^{p}
\, dx\right)^{\frac{1}{p}}.
\end{eqnarray}
Now, we distinguish between the case 
$\frac{q-p}{p}\frac{rn}{r-n}<1$ and the case $\frac{q-p}{p}\frac{rn}{r-n}=1$.
\\
{\bf Case} $\mathbf{\frac{q-p}{p}\frac{rn}{r-n}<1{\,\Longleftrightarrow\,\frac{q}{p}<1+\frac{1}{n}-\frac{1}{r}}
}$ 
\\
In this case, as long as $\log(e+t)\ge 1$, estimate \eqref{stimaquasifinbis} implies
\begin{eqnarray*}\label{stimaquafinbis}
	||V(Du)||_{L^{\infty}(B_{\tilde\rho_0})} &\le&  {\|V(Du)\|^{\frac{q-p}{p}\frac{rn}{r-n}}_{L^\infty\left(B_{ \tilde R_0}\right)}}   \left(\int_{B_{  R_0}}V(Du)^{p}
\, dx\right)^{\frac{1}{p}}
\cr\cr
&+& {C\frac{  \|V(Du)\|^{\frac{n(q-p)}{p}}_{L^\infty\left(B_{\tilde R_0}\right)} }{(\tilde R_0 - \tilde \rho_0)^{\frac{n}{p}}} } \cdot\left(\int_{B_{ R_0}}V(Du)^{p}
\, dx\right)^{\frac{1}{p}}.
%&\le& \frac{C \Big(1+\|V(Du)\|^{{\frac{q-p}{p}\frac{rn}{r-n}}}_{L^\infty\left(B_{ \tilde R_0}\right)}\Big)}{( \tilde R_0- \tilde \rho_0 )^{\frac{n}{p}}}  \left(\int_{B_{ \tilde R_0}}V(Du)^{p}
%\, dx\right)^{\frac{1}{p}}
%\cr\cr 
%&&+\frac{C }{( \tilde R_0- \tilde \rho_0 )^{\frac{n}{p}}}  \left(\int_{B_{ \tilde R_0}}V(Du)^{p}
%\, dx\right)^{\frac{1}{p}}.
\end{eqnarray*}
{In the first term in the right hand side  we use} Young's inequality with exponents $\left(\frac{p(r-n)}{rn(q-p)}, \frac{p(r-n)}{p(r-n)-rn(q-p)}\right)$, which is legitimate in this case,  while in the second term use Young's inequality    with exponents $\left(\frac{p}{n(q-p)}, \frac{p}{p-n(q-p)}\right)$  to deduce that
\begin{eqnarray}\label{stimaquasifinbisbis}
	||V(Du)||_{L^{\infty}(B_{\tilde\rho_0})} &\le& \frac12 \|V(Du)\|_{L^{\infty}(B_{\tilde R_0})}+C\left(\int_{B_{ R_0}}V(Du)^{p}
\, dx\right)^{{\frac{r-n}{p(r-n)-rn(q-p)}}} \cr\cr
    &&+\frac{C }{( \tilde R_0- \tilde \rho_0 )^{{\frac{n}{p-n(q-p)}}}} \left(\int_{B_{ R_0}}V(Du)^{p}
\, dx\right)^{{\frac{1}{p-n(q-p)}}}
\end{eqnarray}
and, since previous estimate holds true for every $\rho_0<\tilde \rho_0<\tilde R_0<R_0$, the iteration Lemma \ref{lem:Giaq} concludes the proof {in this case}.
\\
{\bf Case} $\mathbf{\frac{q-p}{p}\frac{rn}{r-n}=1{\,\Longleftrightarrow\,\frac{q}{p}=1+\frac{1}{n}-\frac{1}{r}}}$ 
\\
In this case, estimate \eqref{stimaquasifinbis} reads as
\begin{eqnarray}\label{stimaquasifinter}
	||V(Du)||_{L^{\infty}(B_{\tilde\rho_0})} &\le& C\left(\frac{\|V(Du)\|_{L^{\infty}(B_{ \tilde R_0})}}{\log^{{\frac{\alpha n}{2p(r-n)}}}\left(e+L^{-1} \Big(\|V(Du)\|^{2n(q-p)}_{L^{\infty}(B_{\tilde R_0})}\Big)\right)} \right) \left(\int_{B_{ R_0}}V(Du)^{p}
\, dx\right)^{\frac{1}{p}}\cr\cr
 &+& {\frac{ C \|V(Du)\|^{\frac{n(q-p)}{p}}_{L^\infty\left(B_{ \tilde R_0}\right)} }{(\tilde R_0 - \tilde \rho_0)^{\frac{n}{p}}} } \left(\int_{B_{  R_0}}V(Du)^{p}
\, dx\right)^{\frac{1}{p}}.
%\cr\cr
%&\le&\frac{C\|V(Du)\|_{L^{\infty}(B_{\tilde R_0})}}{\log^{{\frac{\alpha n}{2p(r-n)}}}\left(e+L^{-1} \Big(1+\|V(Du)\|^{2n(q-p)}_{L^{\infty}(B_{\tilde R_0})}\Big)\right)}  \left(\int_{B_{\tilde R_0}}V(Du)^{p}
%\, dx\right)^{\frac{1}{p}}\cr\cr
% &+& \tcr{\frac{ C \Big(1+\|V(Du)\|^{\frac{n(q-p)}{p}}_{L^\infty\left(B_{ \tilde R_0}\right)} \Big)}{(\tilde R_0 - \tilde \rho_0)^{\frac{n}{p}}} } \left(\int_{B_{ \tilde R_0}}V(Du)^{p}
%\, dx\right)^{\frac{1}{p}}+C.
\end{eqnarray}
Define the function $\widehat{L}:[0,+\infty)\to [0,+\infty)$ by setting
\begin{equation}\label{defL}
    \widehat{L}(s)=:\log^{{\frac{\alpha n}{2p(r-n)}}}\left(e+L^{-1} \Big(s^{2n(q-p)}\Big)\right)
\end{equation}
and note that $\widehat{L}$ is a continuous strictly increasing function  by virtue of the definition of $L$ and since $r>n$. Hence %\textcolor{teal}{mettere il riferimento a $s_0$ come nel teorema}
\begin{eqnarray*}
	||V(Du)||_{L^{\infty}(B_{\tilde\rho_0})} &\le & C\frac{\|V(Du)\|_{{L^{\infty}(B_{\tilde R_0})}}}{\widehat{L}\Big(\|V(Du)\|_{L^{\infty}(B_{ \tilde R_0})}\Big)}  \left(\int_{B_{ R_0}}V(Du)^{p}
\, dx\right)^{\frac{1}{p}}\cr\cr
&&\qquad + {C\frac{  \|V(Du)\|^{\frac{n(q-p)}{p}}_{L^\infty\left(B_{ \tilde R_0}\right)} }{(\tilde R_0 - \tilde \rho_0)^{\frac{n}{p}}} } \left(\int_{B_{  R_0}}V(Du)^{p}
\, dx\right)^{\frac{1}{p}}
\cr\cr
&\le&  C [(\widehat{L}^{-1})^*]^{-1}\Big(1+{\|V(Du)\|_{L^{\infty}(B_{\tilde R_0})}}\Big) \left(\int_{B_{ R_0}}V(Du)^{p}
\, dx\right)^{\frac{1}{p}}\cr\cr
&&\quad + {C\frac{ \|V(Du)\|^{\frac{n(q-p)}{p}}_{L^\infty\left(B_{ \tilde R_0}\right)} }{(\tilde R_0 - \tilde \rho_0)^{\frac{n}{p}}} } \left(\int_{B_{  R_0}}V(Du)^{p}
\, dx\right)^{\frac{1}{p}},
\end{eqnarray*}
where, as long as $\widehat{L}^{-1}$ is increasing and satisfies \eqref{assumlemma} by \eqref{esplicitaL} below, we used Lemma \ref{GGP} with $\widehat{L}^{-1}$ in place of ${\varphi}$ and $[(\widehat{L}^{-1})^*]^{-1}(t)$ in place of ${\psi}$, {after having defined appropriately $\widehat{L}^{-1}(\sigma)$, for $\sigma\in [0,1)$}. The function $[(\widehat{L}^{-1})^*]^{-1}(t)$
is clearly concave as it is the inverse of a polar function which is always convex. 
 Using  Young's inequality at \eqref{Fenchel} with $\widehat{L}^{-1}$ in place of $\Phi$, that is
 $$st = {\frac{1}{4} s \, (4 \, t)}\le (\widehat{L}^{-1})^*\left({\frac{1}{4} s}\right)+ \widehat{L}^{-1}({4 t}),$$ {to estimate the first term in the right hand side of previous inequality} yields
\begin{eqnarray*}
	||V(Du)||_{L^{\infty}(B_{\tilde\rho_0})} &\le&  \frac14||V(Du)||_{L^{\infty}(B_{\tilde R_0})}+ \widehat{L}^{-1}\Bigg(C\left(\int_{B_{ R_0}}V(Du)^{p}
\, dx\right)^{\frac{1}{p}}\Bigg) \cr\cr
 &+& {\frac{ C \|V(Du)\|^{\frac{n(q-p)}{p}}_{L^\infty\left(B_{ \tilde R_0}\right)} }{(\tilde R_0 - \tilde \rho_0)^{\frac{n}{p}}} } \left(\int_{B_{  R_0}}V(Du)^{p}
\, dx\right)^{\frac{1}{p}}.
\end{eqnarray*}
{To estimate the last term, we use Young's inequality    with exponents $\left(\frac{p}{n(q-p)}, \frac{p}{p-n(q-p)}\right)$ thus obtaining \begin{eqnarray*}
	||V(Du)||_{L^{\infty}(B_{\tilde\rho_0})} &\le&  \frac12||V(Du)||_{L^{\infty}(B_{\tilde R_0})}+ \widehat{L}^{-1}\Bigg(C\left(\int_{B_{ R_0}}V(Du)^{p}
\, dx\right)^{\frac{1}{p}}\Bigg) \cr\cr
 &+& \frac{ C }{(\tilde R_0 - \tilde \rho_0)^{\frac{n}{p-n(q-p)}}}  \left(\int_{B_{  R_0}}V(Du)^{p}
\, dx\right)^{\frac{1}{p-n(q-p)}}.
\end{eqnarray*}}
Hence, we may use  Lemma  \ref{lem:Giaq} to get
\begin{eqnarray}\label{apriori}
	||V(Du)||_{L^{\infty}(B_{\rho_0})} &\le&   C\widehat{L}^{-1} \Bigg(C\left(\int_{B_{ R_0}}V(Du)^{p}
\, dx\right)^{\frac{1}{p}}\Bigg)\cr\cr
&&\quad + {\frac{ C }{( R_0 -  \rho_0)^{\frac{n}{p-n(q-p)}}}  \left(\int_{B_{  R_0}}V(Du)^{p}
\, dx\right)^{\frac{1}{p-n(q-p)}}}.
\end{eqnarray}
By the definition of $\widehat{L}$ at \eqref{defL}, we can easily get the expression of $\widehat{L}^{-1}(\sigma)$ arguing as follows
$$\widehat{L}^{-1}(\sigma)=s \Longleftrightarrow\widehat{L}(s)=\sigma \Longleftrightarrow \sigma=\log^{{\frac{\alpha n}{2p(r-n)}}}\left(e+L^{-1} \Big(s^{2n(q-p)}\Big)\right)$$
i.e.
$$\exp\left(\sigma ^{{\frac{2p(r-n)}{\alpha n}}}\right)=e+L^{-1} \Big(s^{2n(q-p)}\Big)\Longleftrightarrow s^{2n(q-p)}=L\left(\exp\left(\sigma ^{{\frac{2p(r-n)}{\alpha n}}}\right)-e\right)$$
that yields the desired expression of $\widehat{L}^{-1}(\sigma)$
\begin{equation*}
\widehat{L}^{-1}(\sigma)=\left[L\left(\exp\left(\sigma ^{{\frac{2p(r-n)}{\alpha n}}}\right)-e\right)\right]^{\frac{1}{2n(q-p)}}.    
\end{equation*}
{Further, recalling that $L(t)=t^{r-n}\log^\alpha(e+t)$, we have
\begin{eqnarray}\label{esplicitaL}
\widehat{L}^{-1}(\sigma)&=& \Big[\Big(\exp\big(\sigma^{\frac{2p(r-n)}{\alpha n}}\big)-e\Big)^{r-n}\log^\alpha\Big(\exp(\sigma^{\frac{2p(r-n)}{\alpha n}})-e+e\Big)\Big]^{\frac{1}{2n(q-p)}} \cr\cr 
&=& \Big[\Big(\exp\big(\sigma^{\frac{2p(r-n)}{\alpha n}}\big)-e\Big)^{r-n}\Big(\sigma^{\frac{2p(r-n)}{\alpha n}}\Big)^\alpha\Big]^{\frac{1}{2n(q-p)}}\cr\cr
&=&\Big(\exp\big(\sigma^{\frac{2p(r-n)}{\alpha n}}\big)-e\Big)^{\frac{r-n}{2n(q-p)}}\sigma^{\frac{p(r-n)}{n^2(q-p)}}\cr\cr
&=&\left[\Big(\exp\big(\sigma^{\frac{2p(r-n)}{\alpha n}}\big)-e\Big)\sigma^{\frac{2p}{n}}\right]^{\frac{r-n}{2n(q-p)}}.
\end{eqnarray}
Since $\frac{q-p}{p}\frac{rn}{r-n}=1$  we have $\frac{r-n}{n}=r\frac{q-p}{p}$ and so
\begin{equation}\label{definversaL}
    \widehat{L}^{-1}(\sigma)\le\left(\exp\Big(\sigma^{\frac{2p(r-n)}{\alpha n}}\Big)\sigma^{\frac{2p}{n}}\right)^{\frac{r}{2p}}=\exp\Big(\frac{r}{2p}\cdot\sigma^{\frac{2p(r-n)}{\alpha n}}\Big)\sigma^{\frac{r}{n}}.\end{equation}}
Inserting \eqref{definversaL} in \eqref{apriori} we get
{\begin{eqnarray*}\label{apriorifin}
	||V(Du)||_{L^{\infty}(B_{\rho_0})} &\le&   C\left(\exp\Bigg(C\left(\int_{B_{ R_0}}V(Du)^{p}
\, dx\right)^{\frac{2(r-n)}{\alpha n}}\Bigg)\right)\cr\cr
&&\quad\cdot   \left(\int_{B_{R_0}}V(Du)^{p}
\, dx\right)^{\frac{r}{np}}\cr\cr
&&\qquad + {\frac{ C }{( R_0 -  \rho_0)^{\frac{n}{p-n(q-p)}}}  \left(\int_{B_{ \tilde R_0}}V(Du)^{p}
\, dx\right)^{\frac{1}{p-n(q-p)}}},
\end{eqnarray*}}
\begin{comment}
\begin{eqnarray*}\label{apriorifin}
	||V(Du)||_{L^{\infty}(B_{\rho_0})} &\le&   \left[L\left(\exp\Bigg(\frac{C}{( R_0- \rho_0 )^{\frac{2(r-n)}{\alpha n}}}\left(\int_{B_{ R_0}}V(Du)^{p}
\, dx\right)^{\frac{2(r-n)}{\alpha n^2}}\Bigg)\right)\right]^{\frac{1}{2n(q-p)}}\cr\cr
&&\quad+ \frac{C}{(R_0-\rho_0 )^{\frac{n}{p}}}  \left(\int_{B_{R_0}}V(Du)^{p}
\, dx\right)^{\frac{1}{p}},
\end{eqnarray*}
\end{comment}
i.e. the conclusion.
 \end{proof}
 \section{The approximation}

\label{quattro}

In this section we give only a sketch of  the approximation procedure that comes along quite standard arguments (in view of the assumption $F(x, \xi) = \tilde{F}(x, |\xi|)$).
Let us start recalling an approximation Lemma whose proof can be achieved arguing exactly as in   \cite[Lemma 4.1]{CupGuiMas} (see also \cite[Lemma 4.2]{CGGP}}).
\begin{lemma}\label{approx} Let $F:\Omega\times \mathbb{R}^n\to [0,+\infty)$ be a Carath\'eodory function satisfying assumptions {\eqref{(F2)(F3)}--\eqref{(F4)}} {with the function $h$ satisfying} {\eqref{F5F6}} in Section 1.  Then there exists a sequence of Carath\'eodory functions $F_k:\,\Omega\times \mathbb{R}^n\to [0,+\infty)$ monotonically convergent to $F(x,\xi)$ such that
\medskip

i) $F_k(x, \xi)=\tilde F_k(x, |\xi|)$, for a.e. $x\in \Omega$ and every $\xi\in \mathbb{R}^n$

ii) $F_k(x, \xi)\le  F_{k+1}(x, \xi)\le F(x,\xi)$, for a.e. $x\in \Omega$ , every $\xi\in \mathbb{R}^n$ and  $k\in\mathbb{N}$

iii) there exist $\ell$ independent of $k$ and $\ell_1$ depending also on $k$ such that
\begin{equation}
\label{(A1)}
\left \{
\begin{array}{lll}
&  \frac{1}{\ell}(\mu^2 + |\xi|^2)^{p/2}-\frac{1}{\ell}\le  F_k(x, \xi) \le \ell (1 +  |\xi|^2)^{q/2} 
\\[1mm]
& F_k(x, \xi) \le \ell_1(k) (\mu^2 + |\xi|^2)^{p/2}  
\end{array}
\right.
\end{equation}
for a.e. $x\in \Omega$, every $\xi\in \mathbb{R}^n$

iv) there exists a constant $\tilde\lambda$ depending only on $\lambda, p$ such that
\begin{equation}
\label{(A2)}
\tilde\lambda \, (\mu^2 + |\xi|^2)^{\frac{p-2}{2}} |\eta|^2 \le \sum_{i,j} (F_k)_{\xi_i \xi_j}(x, \xi) \eta_i \eta_j, 
\end{equation}
 for a.e. $x\in \Omega$, every $\xi\in \mathbb{R}^n$
 
v)  there exist $\tilde\Lambda$ independent of $k$ and $\tilde\Lambda_1$ depending also on $k$ such that
\begin{equation}
\label{(A3)}
|(F_k)_{\xi_i \xi_j}(x, \xi)| \le \, \tilde\Lambda \, (1 + |\xi|^2)^{\frac{q-2}{2}} \quad |(F_k)_{\xi_i \xi_j}(x, \xi)| \le \, \tilde\Lambda_1(k) \, (\mu^2 + |\xi|^2)^{\frac{p-2}{2}}  
\end{equation}
for a.e. $x\in \Omega$ , every $\xi\in \mathbb{R}^n$

vi)  there exist $C$ independent of $k$ and $C_1$ depending also on $k$ such that
\begin{equation}
\label{(A4)}
|(F_k)_{\xi x}(x, \xi)| \le C h(x) (\mu^2 + |\xi|^2)^{\frac{q-1}{2}} \quad |(F_k)_{\xi x}(x, \xi)| \le C_1(k) h(x) (\mu^2 + |\xi|^2)^{\frac{p-1}{2}}  
\end{equation}
for a.e. $x\in \Omega$ , every $\xi\in \mathbb{R}^n$.

\end{lemma}
Next, for a smooth mollifier $\varphi\in C^\infty_0(B_1(0))$, let us consider 
$$F_{\varepsilon,k}(x,\xi)=\int_{B_{1}(0)} F_k(x+\varepsilon y, \xi)\varphi(y)\,dy.$$
It is easy to check that $F_{\varepsilon,k}$ still satisfies \eqref{(A1)}--\eqref{(A4)} of previous Lemma and moreover 

\begin{equation}
\label{(A5)}
\vert D_xD_\xi F_{\varepsilon,k}(x,\xi)\vert \le C_2( K)h_\varepsilon(x) (\mu^2+|\xi |^2)^{\frac{p-1}{2}},
\end{equation}
where $h_\varepsilon$ denotes the usual mollification of $h$ {and it also satisfies assumption \eqref{F5F6} by Jensen's inequality.}
\\
Now {fix a ball $B_{R_0}\Subset \Omega$ and} for $u\in W^{1,p}(\Omega)$ a {local} minimizer of the functional {at \eqref{functional}}, let $v_{\varepsilon,k}\in u+W^{1,p}_0(B_{R_0})$ be the unique solution  to the problem
$$\min\left\{\widehat{F}_{\varepsilon,k} (v,B_{R_0}):\,\,\, v\in u+W^{1,p}_0(B_{R_0})\right\},$$
where $$\widehat{F}_{\varepsilon,k} (v,B_{R_0})=\int_{B_{R_0}}F_{\varepsilon,k}(x, Dv)\,dx.$$
By virtue of \eqref{(A1)}--\eqref{(A5)}, we  may use classical regularity results {(see  \cite{AF} for the case $1<p<2$ and \cite{GM} for the case $p\ge 2$)} to deduce that  $v_{\varepsilon,k}\in W^{1,\infty}_{\rm loc}(B_{R_0})$ and that $(\mu^2+|Dv_{\varepsilon,k}|^2)^{\frac{p-2}{4}}Dv_{\varepsilon,k}\in W^{1,2}_{\rm loc}(B_{R_0})$ for every $\varepsilon>0$ and $k\in \mathbb{N} $.
The left inequality in \eqref{(A1)} implies that
\begin{eqnarray}
 \int_{B_{R_0}}(\mu^2+|Dv_{\varepsilon,k}|^2)^{\frac{p}{2}}\, dx&\le& \ell \int_{B_{R_0}} {\Big(1 + } F_{\varepsilon,k} (x,Dv_{\varepsilon,k})\Big) \,dx \label{normaLp} \\[2mm]
&\le& \ell \int_{B_{R_0}} {\Big(1 + } F_{\varepsilon,k} (x,Du) \Big )\,dx,  \nonumber
\end{eqnarray}
where we used the minimality of $v_{\varepsilon,k}$ {with $u$ as an admissible test function}.
Since $F_{k} (x,Du)\in L^1(B_{R_0})$, for every $k\in \mathbb{N}$, taking the limit as $\varepsilon\to 0$, by the properties of the convolutions, we get
\[
\limsup_{\varepsilon\to 0}\int_{B_{R_0}}(\mu^2+|Dv_{\varepsilon,k}|^2)^{\frac{p}{2}}\, dx \le  \ell \int_{B_{R_0}}\Big(1+F_{k} (x,Du)\Big)\,dx \le \ell \int_{B_{R_0}}\Big(1+F (x,Du)\Big)\,dx,
\] 
where the last inequality is due to ii) of Lemma \ref{approx}.
Therefore $\big(v_{\varepsilon,k}\big)_\varepsilon$ is a bounded sequence in $u+W^{1,p}_0(B_{R_0})$ and so there exists $v_k\in u+W^{1,p}_0(B_{R_0})$ such that \begin{equation}\label{weakLp}
  v_{\varepsilon,k} \rightharpoonup v_k \qquad \text{in}\,\, u+W^{1,p}_0(B_{R_0}),  
\end{equation}
{as $\varepsilon\to 0$.} The weak lower semicontinuity of the $L^p$ norm implies
\[
\int_{B_{R_0}}(\mu^2+|Dv_{k}|^2)^{\frac{p}{2}}\, dx\le \liminf_{\varepsilon\to 0}\int_{B_{R_0}}(\mu^2+|Dv_{\varepsilon,k}|^2)^{\frac{p}{2}}\, dx \le \ell\int_{B_{R_0}}(1+F (x,Du))\,dx 
\]
and hence 
there exists $v\in u+ W^{1,p}(B_{R_0})$ such that \begin{equation}\label{weakLpbis}
  v_{k} \rightharpoonup v \qquad \text{in}\,\, u+W^{1,p}(B_{R_0}),  
\end{equation}
{as $k\to +\infty$.} On the other hand, since $v_{\varepsilon,k}\in W^{1,\infty}_{\rm loc}(B_{R_0})$ and $(\mu^2+|Dv_{\varepsilon,k}|^2)^{\frac{p-2}{4}}Dv_{\varepsilon,k}\in W^{1,2}_{\rm loc}(B_{R_0})$ {and the energy density $F_{\varepsilon, k}$ satisfies \eqref{(A1)}--\eqref{(A5)},} we can use the a priori estimate at \eqref{mainaptotale} to deduce that
\begin{eqnarray}\label{apriori2}
||V(Dv_{\varepsilon,k})||_{L^\infty \left (B_{\frac{R_0}{2}}\right )}&\le &  \textcolor{black}{G \left( C \int_{B_{R_0}}  \, (1 + |Dv_{\varepsilon,k}|^2)^{\frac{p}{2}} dx   \right)}\cr\cr
&&\qquad + \frac{C}{ R_0 ^{\frac{n}{p-n(q-p)}}}  \left(\int_{B_{  R_0}}\, (1 + |Dv_{\varepsilon,k}|^2)^{\frac{p}{2}}  
\, dx\right)^{\frac{1}{p-n(q-p)}}
\cr\cr 
&\le&  \textcolor{black}{G\left(C\int_{B_{R_0}} \Big(1+F (x,Du)\Big)\,dx\right)}\cr\cr
&&\qquad + \frac{C}{ R_0 ^{\frac{n}{p-n(q-p)}}}  \left(\int_{B_{  R_0}} \Big(1+F (x,Du)\Big)
\, dx\right)^{\frac{1}{p-n(q-p)}},
\end{eqnarray}
 where, in the last inequality, we used \eqref{normaLp} and the monotonicity of the function $G$ defined in \eqref{DefinizioneG}.
 \\
Therefore, passing to the limit first as $\varepsilon\to 0$ and then as $k\to +\infty$, we infer that
\begin{equation}\label{weakLpbi}
  v_{\varepsilon,k} \stackrel{*}\rightharpoonup v \qquad \text{weakly * in}\,\, W^{1,\infty}_{\mathrm{loc}}(B_{R_0})  
\end{equation}
and, by the lower semicontinuity of the norm, also
\begin{eqnarray}\label{apriori2bis}
||V(Dv)||_{L^\infty \left(B_{\frac{R_0}{2}}\right )}
&\le& \textcolor{black}{G\left(C\int_{B_{R_0}}\Big(1+F (x,Du)\Big)\,dx\right)}\cr\cr
&&\qquad + \frac{ C }{ R_0 ^{\frac{n}{p-n(q-p)}}}  \left(\int_{B_{  R_0}} \Big(1+F (x,Du)\Big)
\, dx\right)^{\frac{1}{p-n(q-p)}}. 
\end{eqnarray}
The next step consists in 
 showing that $v=u$ a.e. in $B_{R_0}$, and this follows by  the strict convexity of the map $\xi\to F(x,\xi)$ through standard arguments that can be found, for instance, in \cite{EMM16,CGGP,CGHPdN20}.
 At this point, the conclusion of the proofs of Theorems \ref{mainresult}, \ref{mainresult2} and \ref{mainresult3} follows in view of the definition of the function $G(t)$ at \eqref{DefinizioneG}.

\vspace{10mm}

\noindent {\bf Acknowledgments.} The authors are members of the Gruppo Nazionale per l’Analisi Matematica,
la Probabilità e le loro Applicazioni (GNAMPA) of the Istituto Nazionale di Alta Matematica (INdAM). 
A. Passarelli di Napoli has been partially supported through the INdAM-GNAMPA Project 2025 (CUP: E5324001950001) ``Regolarità di soluzioni di equazioni paraboliche a crescita nonstandard degeneri''.
 M. Eleuteri and  A. Passarelli di Napoli have been partially supported through the INdAM-GNAMPA Project 2026 (CUP E53C25002010001) ``Esistenza e regolarità per soluzioni di equazioni ellittiche e paraboliche anisotrope''.

\end{document}